\documentclass[a4paper,12pt]{amsart} 
\title{Generalized Minkowski weights and Chow rings of T-varieties}
\author{Ana Mar\'ia Botero}
\usepackage{amsmath, amssymb, mathrsfs, amsthm, geometry}
\usepackage{amsmath}
\usepackage{bbm}
\usepackage{latexsym}
\usepackage[english]{babel}
\usepackage{amsfonts}
\usepackage[T1]{fontenc}
\usepackage[utf8]{inputenc}
\usepackage{amsthm}
\usepackage{amssymb}
\usepackage{dsfont}
\usepackage{version}
\usepackage{caption}
\usepackage[backend=bibtex,style=alphabetic,hyperref=true,maxnames=9]{biblatex}
\usepackage[psamsfonts]{eucal}
\usepackage{stmaryrd}
\usepackage{url}
\usepackage{hyperref}
\usepackage{color}
\usepackage{blindtext}
\usepackage{tikz}
\usetikzlibrary{arrows,matrix,positioning}
\tikzstyle{dmatrix}=[matrix of math nodes,row sep=2.5em, column sep=2.5em,
text height=1.5ex, text depth=0.25ex] 
\usepackage{mathtools}
\usepackage{microtype}
\usepackage{graphicx}
\usepackage{mathabx}
\usepackage{lipsum}
\usepackage{float}
\usepackage{csquotes}
\usepackage{empheq}
\usepackage{enumitem}
\usepackage{mathrsfs}
\usepackage{mathtools}
\DeclareSymbolFont{epsilon}{OML}{ntxmi}{m}{it}
\DeclareMathSymbol{\epsilon}{\mathord}{epsilon}{"0F}

\theoremstyle{plain}
\newtheorem{theorem}{Theorem}[section]
\newtheorem{lemma}[theorem]{Lemma}
\newtheorem{prop}[theorem]{Proposition}
\newtheorem{cor}[theorem]{Corollary}

\newtheorem{prop/Def}[theorem]{Propsition/Definition}
\newtheorem{theorem/Def}[theorem]{Theorem/Definition}
\theoremstyle{definition}

\newtheorem{Def}[theorem]{Definition}
\newtheorem{rem}[theorem]{Remark}
\newtheorem{exa}[theorem]{Example}
\newtheorem{notation}[theorem]{Notation}

\def \H {{\mathbb H}}
\def \Q {{\mathbb Q}}
\def \R {{\mathbb R}}

\def \Z {{\mathbb Z}}
\def \C {{\mathbb C}}
\def \O {{\mathcal O}}
\def \P {{\mathbb P}}

\def \S {{\mathcal S}}
\def \rec {{\operatorname{rec}}}
\def \div {{\operatorname{div}}}
\def \SF {{\operatorname{SF}}}
\def \Hom {{\operatorname{Hom}}}
\def \CaSF {{\operatorname{CaSF}}}
\def \Spec {{ \operatorname{Spec}}}
\def \Eff {{ \operatorname{Eff}}}

\def \orb {{ \operatorname{orb}}}
\def \codim {{ \operatorname{codim}}}
\def \CaDiv {{ \operatorname{CaDiv}}}
\def \deg{{ \operatorname{deg}}}
\def \trop{{ \operatorname{trop}}}

\newcommand{\ca}[1]{{\mathcal{#1}}}
\newcommand{\on}[1]{{\operatorname{#1}}}

\renewbibmacro{in:}{}

\usepackage{tgpagella}
\usepackage{tgheros}
\usepackage{eulervm}
\usepackage{tikz,tikz-cd}
\usepackage{etoolbox}
\usepackage{subfigure}
\usetikzlibrary{arrows,matrix,positioning}
\tikzstyle{dmatrix}=[matrix of math nodes,row sep=2.5em, column sep=2.5em,
text height=1.5ex, text depth=0.25ex] 
\usepackage{color}

\numberwithin{equation}{section}

\bibliography{bibliography_update}

\date{}
\subjclass{14C25, 14L30, 14M25, 14C15, 14T90}
 \thanks{This work was partially supported by the Deutsche Forschungsgemeinschaft (DFG, German Research Foundation) –
Project-ID 491392403 – TRR 358.}
\begin{document}

\maketitle

{\small{\begin{abstract}
We give a combinatorial characterization of Fulton's operational Chow cohomology  groups of a complete , $\Q$-factorial, rational T-variety of complexity one in terms of so called generalized Minkowsky weights in the contraction-free case. We also describe the intersection product with Cartier invariant divisors in terms of the combinatorial data. In particular this provides a new way of computing top intersection numbers of invariant Cartier divisors combinatorially.
\end{abstract}}}

\tableofcontents

\section{Introduction}
Any algebraic variety $X$ has Chow ``homology'' groups $A_*(X) = \bigoplus_kA_k(X)$ and Chow ``cohomology'' groups $A^*(X) = \bigoplus_kA^k(X)$. The latter are the operational groups defined in \cite[Chapter 17]{fulint}. These cohomology groups have a natural graded ring structure, written with a cup product ``$\cup$'', and $A_*(X)$ is a module over $A^*(X)$, written with a cap product ``$\cap$''. 

If $X$ is complete, composing with the degree homomorphism from $A_0(X)$ to $\Z$ one has a a Kroenecker duality homomorphism 
\[
A^k (X) \longrightarrow \on{Hom}\left(A_k(X),\Z\right), \; z \longmapsto \left(a\mapsto \on{deg}(z \cap a)\right).
\]
In general, this map is not an isomorphism.

However, we will see that for a large class of T-varieties of complexity one, the Kroenecker duality homomorphism is indeed an isomorphism, at least after tensoring with $\Q$. This will follow from Poincaré duality and a combinatorial property related to the T-variety called shellability (see Proposition \ref{prop:kron-dual}). 
%Indeed, it follows from the finite presentation of their Chow groups that such varieties satisfy the Künneth formula (see \ref{prop:kunneth}). Then, as in \cite[Theorem 3]{FMSS}, it is a formal consequence that the Kronecker duality is an isomorphism. 

We now briefly recall the definition and combinatorial characterization of T-varieties. We refer to \cite{AIPSV}, \cite{AH}, \cite{AHS} and \cite{PS} for further details on general T-varieties, and to \cite{LLM}, \cite{IS} and \cite{UN} for the case of complexity one.

%However, it has been shown in \cite[Theorem 3]{FMSS} that if a connected solvable linear algebraic group acts on $X$ with only finitely many orbits, then the Kroenecker duality homomorphism is indeed  an isomorphism. This applies in particular to T-varieties. We now briefly recall their definition and combinatorial characterization. 

A T-variety is a normal algebraic variety $X$ defined over an algebraically closed field of characteristic zero with an effective action of an algebraic torus T. The complexity of a T-variety is defined as $\dim X -\dim T$. Hence T-varieties of complexity zero correspond exactly to toric varieties. As for toric varieties, one can encode T-varieties of arbitrary complexity in terms of some combinatorial data. For this, one uses the language of p-divisors and divisorial fans (see Section~\ref{sec:prel-t-var} for details). In the case of complexity one, one even has a simpler description in terms of so called fansy divisors. This case is in some ways close to the case of toric varieties. For example, one has a characterization of invariant Cartier divisors in terms of so called divisorial Cartier support functions (see Section~\ref{sec:cartier-div}). These are collections of piecewise-affine functions defined on the slices of the divisorial fan sharing the same recession function. Such a characterization is not known in higher complexity. This is one of the reasons why we decided to restrict to the complexity-one case in the present article. 

In order to state our results, let $X$ be a complete, rational T-variety of complexity one with divisorial fan $\S$ and recession fan $\Sigma$. We can think of $\S$ as a collection of complete rational polyhedral complexes $(\S_p)_{p \in \P^1}$, all having recession fan $\Sigma$, and $\S_p \neq \Sigma$ only for finitely many $p \in \P^1$. The $\S_p$ are called the slices of $\S$. Let $P$ be the set of points such that $\S_p \neq \Sigma$. We assume that $|P|\geq 2$. 

It has been shown in \cite{UN} that the Chow groups $A_k(X)$ are generated by invariant subvarieties corresponding to elements in the complexes $\S_p$ for $p \in P$, and in $\Sigma$. Roughly, polyhedra in the $\S_p$ correspond to vertical cycles, cones in $\Sigma$ correspond to horizontal cycles, and one has to keep track of which subvarieties are contracted along the map $r \colon \tilde{X} \to X$ (see Section \ref{sec:prel-t-var} for the definition of the map $r$). Relations are given by the divisors of invariant rational functions on subvarieties of dimension $k+1$. This gives an explicit presentation of the Chow groups $A_k(X)$ by an exact sequence (see Theorem~\ref{th:chow-exact}). 

It now follows from the Kroenecker duality map, that one can associate to a given Chow cohomology class a certain function on the set of polyhedra of the $\S_p$'s and the set of cones in $\Sigma$. We call these functions \emph{generalized Minkowski weights} (see Definition \ref{def:gen-mink}). The name is inspired by \cite{FS}, where the authors study the operational Chow cohomology rings of complete toric varieties and show that they can be identified with so called \emph{Minkowski weights}, which are certain functions on the set of cones of the fan defining the toric variety. This association is bijective whenever the Kroenecker duality map is an isomorphism.

We say that $X$ is contraction-free if $\tilde{X} \simeq X$. In this case, and under some other mild assumptions, the Chow cohomology groups are generated in codimension one, and a Stanley--Reisner type presentation of the Chow groups is given in \cite[Theorem 4]{LLM}. We use this description together with Poincaré duality to show that in this case,  the Kroenecker duality map is an isomorphism after tensoring with $\Q$.

The above is summarized in the following theorem, which is our first main result. 
\begin{theorem}\label{th:intro-trop}(Theorem \ref{thm:trop})
Let $X$ be a complete rational T-variety of complexity one. Then there is a tropicalization map
\[
\on{trop}\colon  A^k(X) \longrightarrow M_k(X)
\]
  between the Chow cohomology groups $A^k(X)$ and the group of $k$-codimensional generalized Minkowski weights $M_k(X)$. If $X$ is projective, $\Q$-factorial and contraction-free, then the induced tropicalization map 
 \[
\on{trop}\colon  A^k(X)_{\Q} \longrightarrow M_k(X)_{\Q},
\] 
where $M_k(X)_{\Q}$ denotes $\Q$-valued generalized Minkowski weights, is a bijection.
\end{theorem}

Let us comment on the above assumptions. $\Q$-factoriality is needed for Poincaré duality to hold and, as the authors in \cite{LLM} remark, their contraction-free assumption is necessary for their result, as they are using that $A^k(X)$ is generated by divisors, a fact which is no longer true in the non-contraction-free case. This is why we need to impose all of these conditions in the above theorem. We however expect that the Kronecker duality map is still an isomorphism also in the not necessarily contraction-free case. 

For the rest of the introduction, assume that $X$ is projective, $\Q$-factorial and contraction-free.  The ring structure on $A^*(X)_{\Q}$ makes the group of generalized Minkowski weights $M_*(X)_{\Q} = \bigoplus_kM_k(X)_{\Q}$ into a commutative ring. The second main result of this article is a combinatorial description of the intersection pairing between Chow cohomology classes and T-invariant Cartier divisors. 

Before we state our result, as was mentioned above, a T-Cartier divisor on $X$ corresponds to a divisorial Cartier support function $h$ on $\S$ (Definition \ref{def:div-car-fun}). We denote by $\on{CaSF}(\S)$ the set of divisorial Cartier support functions on $\S$. In Section~\ref{sec:intersect} we define in the contraction-free case an intersection pairing 
\begin{align}\label{eq:intro-pairing}
\on{CaSF}(\S) \times M_k(X) \longrightarrow M_{k+1}(X),\; (h,c) \longmapsto h \cdot c.
\end{align}
We can think of this pairing as an analogue of the so called \emph{corner locus} of a tropical cycle in tropical geometry.
Our second main result states that this pairing is compatible with the isomorphism $\on{trop}$ of above and hence recovers the algebraic intersection product. 
\begin{theorem}\label{th:intro2}(Theorem \ref{th:comp})
Assume that $X$ is projective, $\Q$-factorial and contraction-free. Let $h \in \CaSF(\S)$ and let $D_h$ its associated Cartier divisor. Set $[D_h] \in A^1(X)_{\Q}$ for its corresponding cohomology class. Then for any $z \in A^k(X)_{\Q}$ we have 
\[
\on{trop}\left([D_h] \cup z \right) = h \cdot \on{trop}(z)
\]
as generalized Minkowski weights in $M_{k+1}(X)_{\Q}$.
\end{theorem}
We expect that the pairing \eqref{eq:intro-pairing} generalizes to the non-contraction-free case such that Theorem \ref{th:intro2} continues to hold. This will be part of a future project.

Now, to any invariant Cartier divisor $D_h$ we associate a measure $\mu_h$ on a real vector space $\on{Tot}(\S)$ by means of the pairing \eqref{eq:intro-pairing}(see Definition \ref{def:measure}). It is supported on the vertices of the polyhedral complexes $\S_p$ and the rays of $\Sigma$. Moreover, $\mu_h$ is positive if $D_h$ is nef.

For any integer $\ell$ and any $p\in P$ we denote by $\S_p(\ell)$ and by $\Sigma(\ell)$ the set of $\ell$-dimensional polyhedra in $\S_p$ and the set of $\ell$-dimensional cones in $\Sigma$, respectively.

As an application of Theorem \ref{th:intro2} we obtain a way to compute the top intersection number of $D_h$ as an integral with respect to the measure $\mu_h$. 
\begin{cor}(Corollary \ref{cor:top-int})
Assume that $X$ is projective, $\Q$-factorial and contraction-free and let $\dim(X)=n+1$. Let $D_h$ be an invariant Cartier divisor associated to a Cartier divisorial support function $h$ with recession function $\rec(h) = \overline{h}$, and let $\mu_h$ be the associated measure. 
Then 
\[
D_h^{n+1} =\int_{\on{Tot}(\S)}-h\mu_h = \sum_{p \in P}\sum_{F \in \S_p(0)}-h(v_F)\mu_h(F) \sum_{\tau \in \Sigma(1)}-\overline{h}(v_{\tau})\mu_h(\tau),
\]
where for $F \in \S_p(0)$,  $v_F = F$, and for $\tau \in \Sigma(1)$, $v_{\tau}$ denotes the primitive vector spanning the ray $\tau$. 
\end{cor}
%We end the article with an example (see Example \ref{exa:top}). 

This article is a first step towards a convex-geometrical description of the intersection theory of nef b-divisors (b stands for birational) on complexity-one T-varieties in the spirit of \cite{BoteroBurgos}. Indeed, we expect that nef b-divisors on such a T-variety can be characterized in terms of some ``concave'' functions on the divisorial fan. We then expect to associate a Monge--Amp\`ere type measure to such a function, as a weak limit of measures of the form $\mu_h$ from above. We plan to pursue these ideas in the future.  

%Let us mention some related literature. In \cite{Gon} the author studies the equivariant Chow cohomology of so called T-linear schemes. Now, a complete T-variety of complexity one is a T-linear scheme hence the results of \emph{loc.~cit.~} can be applied to our setting. However,  the description given in this article is not of combinatorial nature, and it is not clear how to derive from it an explicit presentation of the relations between invariant cycles as given in Theorem~\ref{th:chow-exact}. 
Finally, in \cite{AP} and \cite{AGP} the authors study equivariant operational K-theory and Grothendieck transformations from bivariant operational K-theory to Chow.  These results are later used in \cite{Shah}, where the author describes the operational K-theory of complete toric varieties in terms of so called Grothendieck weights (a K-theoretic analog of the Minkowski weights) and gives a product formula in terms of these weights. It would be interesting to try to characterize the operational K-theory of complete T-varieties of complexity one in terms of some generalized Grothendieck weights, a K-theoretic analog of the generalized Minkowski weights introduced in the present article. 
%They describe the equivariant operational K-theory of any toric variety in terms of piecewise exponential functions on the associated 

\subsection{Outline of the paper}
In section \ref{sec:prel-t-var} we recall some basic facts of T-varieties and the combinatorial framework to describe them. In particular, we recall the definitions of p-divisors and of divisorial fans. In Section \ref{sec:subvarieties} we restrict to the complexity one case and we describe the invariant subvarieties induced by the polyhedra in the slices of the divisorial fan. It turns out that these invariant subvarieties generate the pseudoeffective cone of the T-variety and we give a list of its generators (Theorem \ref{th:eff_cone}). For this we mainly follow \cite{UN}. Then we recall the classification of T-invariant Cartier divisors in terms of Cartier divisorial support functions, following the work in \cite{AIPSV} and \cite{PS}.

In section \ref{sec:ch} we recall the definition of the Chow homology and cohomology groups of an algebraic variety.  In Section \ref{sec:chow} we show that the Kroenecker duality map is an isomorphism for projective, rational, $\Q$-factorial, contraction-free T-varieties after tensoring with $\Q$ (see Proposition \ref{prop:kron-dual}). Then, following \cite{UN}, we give in the general, not necessarily contraction-free case, an explicit presentation of $A_k(X)$ in terms of an exact sequence.
In section \ref{sec:gen-mink} we prove our first main result. We give the definition of the groups of generalized Minkowski weights and we show that they can be identified with the Chow cohomology groups in the projective, $\Q$-factorial, contraction-free case. This yields Theorem \ref{th:intro-trop}.

In Section \ref{sec:intersect} we prove our second main result. We restrict to the $\Q$-factorial, contraction free case and define an intersection pairing between Cartier divisorial support functions and generalized Minkowski weights. We show that this pairing is compatible with the tropicalization isomorphism from Theorem \ref{th:intro-trop}). This gives Theorem~\ref{th:intro2}.

In Section \ref{sec:measure} we associate a measure $\mu_h$ to any Cartier divisorial support function $h$ using the intersection pairing. As an application of Theorem \ref{th:intro2} we obtain a way of computing top intersection numbers of invariant Cartier divisors in terms of this measure. We get Corollary \ref{cor:top-int}. We end the article with an example (see Example \ref{exa:top}).\\

\noindent
\textbf{Conventions:} $k$ denotes an algebraically closed field of characteristic $0$. A variety means an integral, separated scheme of finite type over $k$.  Varieties are assumed to be reduced and irreducible. A subvariety is a closed subscheme which is a variety.
\vspace{0.5cm}\\

\noindent
\emph{Acknowledgments:} We thank an anonymous referee for their careful reading of this manuscript and their many
insightful comments and suggestions.

\section{Preliminaries on T-varieties}\label{sec:prel-t-var}
We use standard notation from toric geometry. Let $N$ be a lattice of dimension $n$, $M = N^{\vee}$ its dual lattice and $T = \on{Spec}(N) \simeq (k^*)^n$ the $n$-dimensional torus. For any ring $R$ we put $N_{R}= N \otimes_{\Z} R$ and $M_{R}= M \otimes_{\Z}R$. 
\begin{Def}
A \emph{$T$-variety} is a normal variety $X$ together with an effective algebraic torus action $T \times X \to X$. Its \emph{complexity} is defined as $\dim X- n$. 
\end{Def}
Thus T-varieties of complexity zero correspond to toric varieties. 

We recall the general framework for describing $T$-varieties of any complexity. For details we refer to \cite{AH}, \cite{AHS} and \cite{AIPSV}. 

We use notations and basic facts regarding convex cones and polyhedra from \cite{ROCK}. In particular, for any polyhedron $\Lambda \in N_{\Q}$, its \emph{recession cone} $\rec(\Lambda)$ is the set of
direction vectors of all rays contained in $\Lambda$
\[
\rec(\Lambda) \coloneqq \left\{v \in N_{\Q} \; \colon \; v + \Lambda \subseteq \Lambda\right\}.
\]
Fix a polyhedral cone $\sigma \subseteq N_{\Q}$. Consider the semigroup (under Minkowski addition) 
\[
\operatorname{Pol}_{\Q}^+(N, \sigma) = \left(\Lambda \subseteq N_{\Q} \;\colon \; \Lambda \text{ is a polyhedron with }\rec(\Lambda) = \sigma \right\}. 
\]
We also allow $\emptyset \in \operatorname{Pol}_{\Q}^+(N, \sigma)$. 

Let $Y$ be a normal variety. We denote by $\CaDiv(Y)$ the group of Cartier divisors on $Y$ and by $\CaDiv_{\Q}(Y)$ the group of $\Q$-Cartier divisors on $Y$. A \emph{polyhedral divisor on} $(Y,N)$ is a formal sum 
\[
\ca D = \sum_{Z}\Lambda_Z \cdot Z,
\]
for $\Lambda_Z \in \operatorname{Pol}_{\Q}^+(N, \sigma)$ and $Z \in \CaDiv_{\Q}(Y)$,  such that only finitely many $\Lambda_Z$ differ from $\sigma$. For a polyhedral divisor $\ca D$ consider the evaluation 
\[
\ca D(u) = \sum_{Z \text{ s.~t. }\Lambda_Z \neq \emptyset} \min\langle\Lambda_Z, u\rangle Z \in \CaDiv_{\Q}(Y)
\]
for $u \in \sigma^{\vee} \cap M$. 
This is a finite sum. 

The polyhedral divisor $\ca D$ on $(Y,N)$ is called a \emph{$p$-divisor} if $\ca D(u)$ is semiample for all $u \in  \sigma^{\vee} \cap M$, as well as big for $u \in \on{relint}(\sigma^{\vee}) \cap M$. 

To a $p$-divisor $\ca D$ one associates a sheaf of rings \[
\ca O_Y(\ca D) = \bigoplus_{u \in \sigma^{\vee}\cap M}\ca O_Y(\ca D(u)).
\]
Then $X = \Spec \;\Gamma \left(Y, \ca O_Y(\ca D)\right)$ is an affine $T$-variety of complexity $\dim Y$. Also the relative spectrum $\tilde{X} = \Spec_Y \ca O_Y(\ca D)$ is a $T$-variety of complexity $\dim Y$ and there is an equivariant map $r \colon \tilde{X} \to X$. The variety $X$ is said the be \emph{contraction-free} if $\tilde{X} = X$.

Altman and Hausen show that any affine $T$-variety arises from a $p$-divisor in this way \cite[Theorem 3.4]{AH}. 

Given two $p$-divisors $\ca D$ and $\ca D'$ on $(Y,N)$ we write $\ca D' \subseteq \ca D$ if each coefficient of $\ca D'$ is contained in the corresponding coefficient of $\ca D$. In this case, we have a map $X(\ca D') \to X(\ca D)$, and $\ca D'$ is said to be a \emph{face} of $\ca D$ if this map is an open embedding. 

The intersection $\ca D \cap \ca D'$ of two polyhedral divisors is the polyhedral divisor $\sum_Z \left(\Lambda_Z \cap \Lambda'_Z\right)\cdot Z$. 
\begin{Def}
A \emph{divisorial fan} $\ca S$ is a finite set of $p$-divisors on $(Y,N)$ such that the intersection of any two $p$-divisors of $\S$ is a face of both and $\S$ is closed under taking intersections. 
\end{Def}
We denote by $X(\S)$ the scheme obtained by gluing the affine $T$-varieties $X(\ca D)$ and $X(\ca D')$ along the open subvarieties $X(\ca D \cap \ca D')$ for each $\ca D$, $\ca D' \in \S$. As in the toric case, $X(\S)$ turns out to be a $T$-variety of complexity $\dim Y$. 

Analogously, define the variety $\tilde{X}(\S)$ by gluing the affine $T$-varieties $\tilde{X}(\ca D)$ and $\tilde{X}(\ca D')$ along the open subvarieties $\tilde{X}(\ca D \cap \ca D')$ for each $\ca D$, $\ca D' \in \S$.

We have an induced proper birational morphism 
\[
r \colon \tilde{X}(\S) \longrightarrow X(\S)
\]
 and $X(\S)$ is said to be \emph{contraction-free} if $r = \on{id}$. The set 
 \[
 \Sigma(\S) = \left\{\rec(\ca D) \; | \; \ca D \in \S\right\}
 \]
 is the \emph{recession fan} of $\S$, where $\rec(\ca D)$ stands for the common recession cone of the coefficients of $\ca D$. 
 
 Altmann, Hausen and Süss show that any $T$-variety arises from a divisorial fan in this way \cite[Theorem 5.6]{AHS}. 
 
The \emph{degree} of a $p$-divisor $\ca D$, denoted $\deg (\ca D)$, is the Minkowski sum $\sum_Z\Lambda_Z$ of its coefficients. 
 
 We now specialize to the case where $Y$ is a smooth projective curve.
\begin{Def}
 A \emph{marked fansy divisor on $Y$} is a formal sum 
 \[
 \Xi = \sum_{p \in Y} \Xi_P \otimes [p],
 \]
 together with a complete fan $\Sigma$ in $N_{\Q}$ and a subset $K \subseteq \Sigma$ such that \begin{enumerate}
 \item for all $p \in Y$, the coefficient $\Xi_p$ is a complete polyhedral subdivision of $N_{\Q}$ with $\rec (\Xi_p) = \left\{\rec(\Lambda_p)\; | \; \Lambda_p \in \Xi_p \right\} = \Sigma$.
 \item for a cone $\sigma \in K$ of full dimension, the element $\ca D^{\sigma} = \sum_p \ca D_p^{\sigma} \otimes [p]$ is a $p$-divisor, where $\ca D_p^{\sigma}$ denotes the unique polyhedron in $\Xi_p$ whose recession cone is equal to $\sigma$. 
 \item for a full-dimensional cone $\sigma \in K$ and a face $\tau \prec \sigma$ we have that $\tau \in K$ if and only if $\deg(\ca D^{\sigma})\cap \tau \neq \emptyset$. 
 \item if $\tau$ is a face of $\sigma$ then $\tau \in K$ implies that $\sigma \in K$. 
 \end{enumerate}
 The elements of $K \subseteq \Sigma$ are called \emph{marked cones}.
 \end{Def}
 It is shown in \cite[Proposition 1.6]{IS} that any complete $T$-variety of complexity one corresponds to a marked fansy divisor. Given a divisorial fan $\S$ one associates a marked fansy divisor $\Xi(\ca S)$ by taking the polyhedral subdivisions $\S_p$ given by $\S$ and letting $K$ consist of all cones $\sigma \in \Sigma$ such that there exists $\ca D \in \S$ with $\rec (\ca D) = \sigma$ and such that no coefficient of $\ca D$ equals $\emptyset$. The variety $\tilde{X}$ is given as a marked fansy divisor by the same subdivisions as for $X$ but with $K = \emptyset$. Thus the cones in $K$ capture the information of which orbits are identified via the contraction map $r$. 
 
 Intuitively, one can think of a $T$-variety of complexity one in the following way. For a point $p \in Y$, the fiber over $p$ corresponds to a union of toric varieties corresponding to the polyhedral subdivision $\S_p$. The polyhedral subdivision at a general fiber is $\Sigma = \rec (\S)$ and hence the generic fiber is the toric variety corresponding to $\Sigma$.

\section{Subvarieties of T-varieties of complexity one}\label{sec:subvarieties}
Let $X$ be a $(n+1)$-dimensional complete, rational T-variety of complexity one with divisorial fan $\S$ over $Y= \P^1$, recession fan $\Sigma$ in $N_{\Q}$ and subset $K \subseteq \Sigma$. Assume that $K$ is a proper subset of $\Sigma$, i.e. $K \neq \Sigma$. 

For $p \in \P^1$ we write $\S_p$ for the corresponding complete polyhedral subdivision of $N_{\Q}$. Let $P$ be the set of points of $\P^1$ such that $\S_p \neq \Sigma$. $P$ is a finite set. We assume that $P$ has size at least two. 

\begin{notation}\label{notation} We use standard notation in toric and tropical geometry. For a cone $\sigma \in \Sigma$ we set $N_{\sigma}= N\cap \on{span}(\sigma)$ and $N(\sigma) = N/N_{\sigma}$. For any integer $k \geq 0$ we denote by $\Sigma(k)$ the set of cones of $\Sigma$ of dimension $k$. Also for a polyhedral subdivision $\S_p$ we denote by $\S_p(k)$ the set of $k$-dimensional polyhedra in $\S_p$. If $\tau \in \Sigma(k)$ is a cone. Then for every cone $\sigma\in \Sigma(k+1)$ with $\tau \prec \sigma$ we denote by $v_{\tau,\sigma}$ the \emph{lattice normal vector of $\sigma$ relative to $\tau$}. This is the image in $N(\tau)$ of the unique generator of $N_{\sigma}/N_{\tau}$ that points in the direction of $\sigma$. If $k=0$ we also write $v_{\sigma}$ for $v_{\{0\},\sigma}$. 

For $F \in \S_p$ denote by $N_{F,\R}$ the linear space spanned by all $x - y$ for $x, y \in \sigma$. Further, denote by $N_F = N_{F,\R} \cap N$ the induced sublattice and by $N(F)$ the quotient $N/N_F$.  If $F \in \S_p(k)$ then for every $G \in \S_p(k +1)$ denote by $v_{F,G}$ the \emph{lattice normal vector of $G$ relative to $F$}. This is the image in $N(F)$ of the unique generator of $N_G/N_F$ that points in the direction of $G$. 

Be aware that the notation $v_F$ for $F \in \S_p$ is given below in Notation \ref{not:vertex}.

%For $k=0$ and $\nu$ a vertex in $\S_p$ we also write $v_F$ instead of $v_{\nu,F}$ for any edge $F$ in $\S_p(1)$ containing $\nu$. 

%For $F \in \S_p$ denote by $\on{c}(F) \subseteq N \oplus \Z$ the cone having $F$ at height $1$ and set $M(F) = N(F)^{\vee}$, where $N(F) = N \oplus\Z /\left((N \oplus Z) \cap \R\on{c}(F)\right)$. If $F \in \S_p(k)$ then for every $G \in \S_p(k +1)$ denote by $v_{F,G}$ the generator of $N(F)/N(G)$. 
\end{notation}
On $\tilde{X}$ we have various subvarieties arising from the combinatorial structure. We briefly recall the description of these here from \cite[Section 3]{UN} (see also \cite{HS-cox}).
\begin{itemize}
\item For $p \in \P^1$ and a polyhedron $F \in \S_p$, there is a T-orbit $\orb(p,F) \subseteq \widetilde{X}$ of dimension $\codim(F)$. The closure of this orbit in $ \widetilde{X}$ is denoted $Z_{p,F}$. It is an invariant subvariety of an irreducible component of the fiber. The irreducible component is a toric variety corresponding to the fan $\Sigma(p,F) \subseteq N(F)_{\Q}$ consisting of all polyhedral in $\S_p$ containing $F$.  The corresponding character lattice is $M(F) = N(F)^{\vee}$.
\item Let $\eta$ be the generic point of $\P^1$. For any cone $\sigma \in \Sigma$ we have a T-orbit $\orb(\eta,\sigma)$ of dimension $\codim(\sigma) + 1$. The closure of this orbit in $ \widetilde{X}$ is denoted by $B_{\sigma}$. Then $B_{\sigma}$ dominates $\mathbb{P}^1$ with general fiber given by the subvariety of $X_{\Sigma}$ corresponding to $\sigma$. $B_{\sigma}$ is itself a T-variety of complexity one (see proof of \cite[Theorem 4.1]{UN}). Hence, $B_{\sigma}$ corresponds to a divisorial fan $\S(\sigma)$. Over each $p \in \mathbb{P}^1$, the polyhedral complexes $\S(\sigma)_p \subseteq N(\sigma)_{\Q}$ consist of all polyhedra in $\S_p$ with recession cone containing $\sigma$. The recession fan of $\S(\sigma)$ is the star $\Sigma(\sigma)$. For example, for $\tau \in \Sigma(1)$, $B_{\tau}$ is a horizontal divisor in the sense of \cite{LLM}.  
\end{itemize}
On $X$, some of the subvarieties are contracted, namely those $B_{\sigma}$ where $\sigma$ is in $K$. In this case, $r(B_{\sigma})$ is contracted to a subvariety of one dimension less. We denote by $W_{\sigma}$ the corresponding orbit closure $r(B_{\sigma})$ in $X$. Then $\dim(W_{\sigma}) = \codim(\sigma)$.

Denote also by the same symbols $Z_{p,F}$ and $B_{\sigma}$ the subvarieties  $r(Z_{p,F})$ and $r(B_{\sigma})$ for $\sigma \notin K$, respectively.

Let us introduce a bit more notation. 

\begin{notation}\label{not:vertex} Let $F \in S_p$ with $\rec(F) = \sigma$. Then from the description of the invariant subvariety $B_{\sigma}$ given above, it follows that the image of $F$ in $\S(\sigma)_p$ is zero-dimensional. We denote this image by $v_F$. It is an element in $N(\sigma)_{\Q}$. Let $\mu(F)$ be the smallest positive integer in $N(\sigma)$ such that $\mu(F)v_F$ is a lattice point. 
\end{notation}

If $\sigma \in K$ we consider the index $s_{\sigma}=\left[\on{Stab}\left(r\left(Z_{p,F}\right)\right) \colon \on{Stab}\left(Z_{p,F}\right)  \right]$ (note that since the map $r$ is equivariant, any $t \in T$ in the stabilizer of $Z_{p,F}$ will be in the stabilizer of its image). 

For any integer $k \geq 0$ let $\overline{\Eff_k}(X) \subseteq N_k(X)$ be the pseudoeffective cone of $X$ inside the group of $k$-cycles modulo numerical equivalence.
The following is \cite[Proposition 3.5]{UN}.
\begin{theorem}\label{th:eff_cone}
For $k \geq 0$, the pseudoeffective cone $\overline{\Eff_k}(X)$ of $X$ is rational polyhedral, generated by classes of invariant subvarieties. Moreover, a list of generators is provided by the following classes.
\begin{itemize}
\item $B_{\sigma}$ with $\sigma \in \Sigma(n-k+1), \; \sigma \notin K$,
\item $Z_{p,F}$ with $p \in \P^1, \; F \in \S_p(n-k),\; \rec(F) \notin K$,
\item $W_{\sigma}$ with $\sigma \in \Sigma(n-k), \; \sigma \in K$.
\end{itemize}
\end{theorem}
\begin{rem}
 Note that in the theorem above, subvarieties of the form $Z_{p,F}$ with $\rec(F) \in K$ do not appear. Indeed, for a cone $\sigma \in K$, one can show that $r_*(Z_{p,F})$ is numerically equivalent to $\left(s_{\rec(F)}/\mu(F)\right)\cdot W_{\rec(F)}$ for any $p \in \P^1$ and any $F$ with $\rec(F) = \sigma$(see \cite[Lemma 3.4]{UN}). Hence all such classes $Z_{p,F}$ are proportional and we need only to remember the representative $W_{\rec(F)}$.
\end{rem}
\subsection{T-Cartier divisors}\label{sec:cartier-div}
We now briefly describe T-invariant Cartier divisors on $X$ in terms of divisorial support functions on $\S$. For details we refer to \cite[Section 7]{AIPSV} and \cite{PS}.

First, recall that given a piecewise affine function $f$ on a polyhedral subdivision $\S_p$ of $N_{\Q}$ with recession fan $\Sigma$, its \emph{recession function} $\on{rec}(f)$ is the piecewise linear function on $\Sigma$ defined by 
\[
\on{rec}(f)\colon |\Sigma| \longrightarrow \R, \; u \longmapsto \lim_{\lambda \to \infty}\frac{f(v_0+\lambda u)-f(v_0)}{\lambda}.
\]
\begin{Def}\label{def:div-car-fun}
A \emph{divisorial support function} on $\S$ is a collection $h = (h_p)_{p \in \P^1}$ of continuous, piecewise affine functions $h_p \colon |\S_p| \to \Q$ such that 
\begin{enumerate}
\item $h_p$ has integral slope and integral translation on every polyhedron in the polyhedral complex $\S_p$ of $N_{\Q}$. In other words, for any $F \in \S_p$,  we have $h_p|_F(u) = \langle m_F, u \rangle +\ell_F$  with $(m_F, \ell_F) \in M \times \Z$.
\item All $h_p$ have the same recession function $\rec(h)$.
\item The set of points $p\in \P^1$ such that $\rec(h) \neq h_p$ is finite.
\end{enumerate}
The set of all divisorial support functions on $\S$ is denoted by $\SF(\S)$.
\end{Def}
To any divisor $D$ on $\P^1$ we can associate a divisorial support function $\SF(D)$ by $\SF(D)_p \equiv \on{coeff}_pD$. Also, for any $u \in M$ we define its corresponding support function $\SF(u)_p\equiv u$. 
\begin{Def}
A divisorial support function $h \in \SF(\S)$ is called \emph{principal} if $h = \SF(u) + \SF(D)$ for some $u \in M$ and some principal divisor $D$ on $\P^1$. It is called \emph{Cartier} if its restriction to any affine T-invariant open subset is principal. We denote by $\CaSF(\S)$ the set of divisorial Cartier support functions. It is a free abelian group. 
\end{Def}
From the description of invariant subvarieties given above we see that there are two classes of $T$-invariant prime Weil divisors on $X$. The \emph{vertical} invariant prime divisors arising as the closure of a family of $n$-dimensional $T$-orbits. These are parametrized by pairs $(p,v)$ of points $p \in \P^1$ and vertices $v$ of the polyhedral subdivision $\S_p$. And then there are the \emph{horizontal} invariant prime divisors arising as the closure of a family of $(n-1)$-dimensional $T$-orbits. These are parametrized by rays $\tau$ of $\Sigma$. 

To a given divisorial support function $h = (h_p)_{p \in \P^1}$, one associates a T-invariant Weil divisor 
\begin{align}\label{eq:weil-div}
D_h = -\sum_{\tau \in \Sigma(1)}h(v_{\tau})B_{\tau} - \sum_{\substack{(p,v)\\ v \in S_p(0)}} \mu(v)h_p(v)Z_{p,v},
\end{align}
where one omits the prime divisors contracted by $r$. 
The following is \cite[Proposition 3.10]{PS}
\begin{prop}
The map 
\[
h \longmapsto D_h
\]
induces an isomorphism of free abelian groups $\CaSF(\S) \simeq \on{T-Ca}(X)$.
\end{prop}
\begin{exa}\label{exa:principal}
If $ h = SF(u)$ for $u \in M$ then 
\[
D_h = \sum_{\tau \in \Sigma(1)}\langle u, v_{\tau}\rangle B_{\tau} - \sum_{\substack{(p,v)\\ v \in S_p(0)}} \mu(v)\langle u,v \rangle Z_{p,v}.
\]
Let $p \neq \infty$ be any point in $\P^1$. If $h = SF\left([p] - [\infty]\right)$ then
\[
D_h= \sum_{ v \in S_{\infty}(0)} -\mu(v) Z_{\infty,v} + \sum_{ v \in S_{p}(0)} \mu(v) Z_{p,v}.
\]
\end{exa}
Consider the Picard group $\on{Pic}(X)$ consisting of all Cartier divisors on $X$ modulo linear equivalence. If $h$ is principal then $D_h$ is $0$ in $\on{Pic}(X)$. Furthermore, if we consider the equivalence relation on $\CaSF(\S)$ given by $h \sim h'$ if and only if $h - h'$ is principal  then it follows as a particular case of \cite[Corollary 3.15]{PS} that the map $h \mapsto D_h$ induces an isomorphism 
\[
\left(\CaSF(\S)/ \sim \right)  \; \simeq \; \on{Pic}(X).
\]
\begin{exa}
It follows from Example \ref{exa:principal} that all special fibers (with corresponding multiplicities) are linearly equivalent. 
\end{exa}

We make two remarks. 
\begin{rem}
It follows from Theorem \ref{th:eff_cone} above that the Picard group $\on{Pic}(X)$ is generated by classes $Z_{p,v}$, for $v \in S_p(0)$, and $B_{\tau}$, for $\tau \in \Sigma(1)$, $\tau \notin K$ since the divisor class corresponding to a contracted ray can be written as a combination of the other classes (see also \cite[Corollary 3.17]{PS}).
\end{rem}

We end this section with the following important definition.
\begin{Def}\label{def:restriction-div}(Restriction to subvarieties)
Let $h \in \CaSF(\S)$. 

For $\tau\in \Sigma$ consider the embedding $\iota \colon B_{\tau} \hookrightarrow \tilde{X}$ of the corresponding subvariety. Recall that $B_{\tau}$ is itself a rational complete $T$-variety of complexity one with divisorial fan $\S(\tau)$. Then $\iota^*D_h$ is an invariant Cartier divisor on $B_{\tau}$ corresponding to a divisorial Cartier support function on $\CaSF(S(\tau))$. We denote this function by $h({\tau})$. 

Similarly, for a point $p \in \P^1$ and $F \in S_p$ such that $Z_{p,F}$ is an invariant subvariety of a toric subvariety $j\colon X_{\Sigma(p,F)} \hookrightarrow \tilde{X}$, we consider the pullback of $j^*D_h$. This gives a toric Cartier divisor and hence a virtual support function on $\Sigma(p,F)$, which we denote by $h_{p,F}$.

\end{Def}

\section{Chow groups}\label{sec:ch}
For any normal, irreducible algebraic variety $Y$ over $k$, the Chow group $A_q(Y)$ of $Y$ is defined to be $Z_q(Y)/R_q(Y)$, where $Z_q(Y)$ is the free abelian group generated by all $q$-dimensional closed subvarieties of $X$, and $R_q(X)$ is the subgroup generated by divisors $[\div(f)]$ of non-zero rational functions $f$ on $(q+1)$-dimensional subvarieties $W$ of $Y$. When an algebraic group $\Gamma$ acts on $Y$ one can form the \emph{$\Gamma$-invariant} Chow group $A_q^{\Gamma}(Y) = Z_q^{\Gamma}(Y)/R_q^{\Gamma}(Y)$, with $Z_q^{\Gamma}(Y)$ the free abelian group generated by $\Gamma$-invariant closed subvarieties of $Y$, and $R_q^{\Gamma}(Y)$ the subgroup generated by all divisors of eigenfunctions on $\Gamma$-invariant $(q+1)$-dimensional subvarieties. Recall that a function $f$ in $R(W)^*$ is said to be an eigenfunction if $g \cdot f = \chi(g)f$ for all $g \in \Gamma$, for some character $\chi = \chi_f$ on $\Gamma$.
The following useful result is \cite[Theorem~1]{FMSS}.
\begin{theorem}\label{th:inv}
If a connected solvable algebraic group $\Gamma$ acts on $Y$, then then the canonical homomorphism $A_k^{\Gamma}(Y) \to A_k(Y)$ is an isomorphism. 
\end{theorem}
We obtain the following corollary. 
\begin{cor}\label{cor:inv}
Let $X$ be a complete rational T-variety of complexity one. Then the Chow groups $A_k(X)$ are generated by the classes described in Theorem \ref{th:eff_cone}. 
\end{cor}
In order to have a complete combinatorial description of the Chow groups $A_k(X)$ of a complete rational T-variety of complexity one one needs to find the relations between these generators. This will be done in Theorem \ref{th:chow-exact} in the following section.

Before, we recall the general definition of Fulton's operational Chow cohomology groups of any normal, irreducible algebraic variety $Y$ as above.

Let $A_*(Y) = \bigoplus_qA_q(Y)$. The Chow groups $A_q(Y)$ are called Chow ``homology'' groups, since they are covariant with respect to proper maps. In order to have a natural intersection product, one considers the Chow ``cohomology'' groups $A^q(Y)$ (see \cite[Chapter~17]{fulint}). They satisfy the following expected functorial properties.
\begin{enumerate}
\item One has ``cup products'' $A^p(Y) \otimes A^q(Y) \to A^{p+q}(Y)$, $a \otimes b \mapsto a \cup b$, making $A^*(Y) =\bigoplus_qA^q(Y)$ into a graded associative ring.
\item There are contravariant graded ring maps $f^* \colon A^*(Y) \to A^*(Y')$ for arbitrary morphisms $f \colon Y' \to Y$;
\item One has ``cap products" $A^q(Y) \otimes A_m(Y) \to A_{m-q}(Y)$, $c \otimes z \mapsto c \cap z$, making $A_*(Y)$ into an $A^*(Y)$-module and satisfying the usual projection formula;
\item when $Y$ is non-singular of dimension $n$, then the ``Poincar\'e duality map'' $A^q(Y) \to A_{n-q}(Y)$, $c \mapsto c \cap [Y]$ is an isomorphism, and the ring structure on $A^*(Y)$ is that determined by the intersection products of cycles of $Y$; 
\item vector bundles on $Y$ have Chern classes in $A^*(Y)$. 
\end{enumerate}
\begin{rem}\label{rem:pd}
The Poincaré duality map is still an isomorphism if one assumes some mild singularities (e.g. $\Q$-factorial), and  after tensoring with $\Q$ (see \cite[Section 3]{KT}). 
\end{rem}
An element in $A^q(Y)$ determines, by the cap product, a homomorphism from $A_q(Y)$ to $A_0(Y)$; and if $Y$ is complete, one can compose with the degree map from $A_0(Y)$ to $\Z$ and one has a natural ``Kroenecker duality'' homomorphism 
\begin{align}
A^q(Y) \longrightarrow \on{Hom}(A_q(Y), \Z), \quad z \longmapsto \left[a \mapsto \on{deg}(z \cap a)\right].
\end{align}
As was mentioned in the introduction, this map is in general not an isomorphism.  We will see in the next section that for projective, $\Q$-factorial, rational, contraction-free T-varieties of complexity one, this map is indeed an isomorphism after tensoring with $\Q$ (Proposition~\ref{prop:kron-dual}). 

\subsection{The Chow group of a complexity-one T-variety}\label{sec:chow}
Let $X$ be a $(n+1)$-dimensional complete, rational T-variety of complexity one with divisorial fan $\S$ over $Y= \P^1$, recession fan $\Sigma$ in $N_{\Q}$ and subset $K \subseteq \Sigma$. Assume that $K$ is a proper subset of $\Sigma$, i.e. $K \neq \Sigma$. We say that a cone $\sigma \in\Sigma$ is \emph{contracted} if $\sigma \in K$.

As before, for $p \in \P^1$ we write $\S_p$ for the corresponding complete polyhedral subdivision of $N_{\Q}$. Let $P$ be the set of points of $\P^1$ such that $\S_p \neq \Sigma$. $P$ is a finite set and we assume that $P$ has size at least two. 

 For an integer $k \geq 0$ we define the following sets. 
\begin{align*}
V_k &= \left\{ F \in \S_p(n-k), \; p \in P \text{ such that }\rec(F) \notin K\right\},\\
R_k &= \left\{ \sigma \in \Sigma(n-k+1), \; \sigma \notin K \right\},\\
T_k &= \left\{ \sigma \in \Sigma(n-k), \; \sigma \in K \right\}.
\end{align*}
\begin{rem}
We always have $T_n = \emptyset$. If $X = \widetilde{X}$ is contraction free, then for all $k$, we have $T_K = \emptyset$. 
\end{rem}

We begin by showing that when $X$ is a projective, $\Q$-factorial, rational, contraction-free $T$-variety of complexity one, then the Kroenecker duality map is in fact an isomorphism after tensoring with $\Q$. 

\begin{prop}\label{prop:kron-dual}
Let $X$ be a projective, $\Q$-factorial, rational, contraction-free $T$-variety of complexity one. After tensoring with $\Q$, the Kroenecker duality map
\begin{align}\label{eq:kron-dual}
A^q(X)_{\Q} \longrightarrow \on{Hom}(A_q(X)_{\Q}, \Q), \quad z \longmapsto \left[a \mapsto \on{deg}(z \cap a)\right]
\end{align}
is an isomorphism. 
\end{prop}
\begin{proof}
By Poincaré duality $A^q(X)_{\Q} \simeq A_{n-q}(X)_{\Q}$ (see Remark \ref{rem:pd}) it suffices to show that the induced pairing
\[
A_{n-q}(X)_{\Q} \times A_q(X)_{\Q} \longrightarrow \Q
\]
is perfect. By \cite[Theorem 4]{LLM}, we have the following Stanley--Reisner type description of the homology groups:
\[ 
A_*(X)_{\Q} = \frac{\Q\left[x_{\omega} \; | \; \omega \in R_n \sqcup V_n\right]}{\ca J},
\]
where $\ca J$ is generated by the linear forms with integer coefficients $\sum n_{\omega}x_{\omega}$ such that the corresponding invariant divisor $\sum_{\omega \in V_n} n_{\omega}Z_{\omega} + \sum_{\omega \in R_n}n_{\omega}B_{\omega}$ is principal, together with the square-free monomials $\prod_{\omega \in R_n \sqcup V_n } x_{\omega}$ such that the intersection $\bigcap_{\omega \in V_n}  Z_{\omega} \cap \bigcap_{\omega \in R_n}B_{\omega}$ is empty.  In particular,  square-free monomials not contained in $\ca J$ are of the form
\[
x_F = \prod_{\tau \in \rec(F)(1)}x_{\tau} \cdot \prod_{v \in F(0)}x_v
\] 
for some $F \in \ca S_p$ (see \cite[Lemma 4.3]{LLM}).

As a key fact to prove this result the authors define a technical combinatorial property called ``shellability'' (\cite[Definition~3.8]{LLM}). The authors show in \cite[Proposition 3.11]{LLM} that the divisorial fans of projective, $\Q$-factorial, contraction-free rational T-varieties of complexity one are shellable, which allows them to prove the above representation.

Now, in order to show the proposition, it thus suffices to show that the pairing 
\begin{equation}\label{eq:perfect}
\left(\frac{\Q\left[x_{\omega} \; | \; \omega \in R_n \sqcup V_n\right]}{\ca J}\right)_{n-q} \times \left(\frac{\Q\left[x_{\omega} \; | \; \omega \in R_n \sqcup V_n\right]}{\ca J}\right)_{q} \longrightarrow \Q,
\end{equation}
where $(\cdot )_{\ell}$ denotes the $\ell$'th graded piece, is perfect.Let $b \in \left(\frac{\Q\left[x_{\omega} \; | \; \omega \in R_n \sqcup V_n\right]}{\ca J}\right)_{n-q}$ be non-zero. Without lost of generality, we can assume that $b = x_F$ for $F \in \ca S_p$. Let $G \in \ca S_p(n)$ be a maximal polyhedron containing $F$ and set 
\[
x_{F'} = \prod_{\rho \in \rec(G)(1) \setminus \rec(F)(1)}x_{\rho} \cdot \prod_{v \in G(0) \setminus F(0)}x_v. 
\]
Since $\ca S_p$ is simplicial we have $x_{F'} \in \left(\frac{\Q\left[x_{\omega} \; | \; \omega \in R_n \sqcup V_n\right]}{\ca J}\right)_{q}$ and $x_{F} \cdot x_{F'} = x_{G} \neq 0$. Hence the pairing in \eqref{eq:perfect} is perfect, which is what we wanted to show. 

\end{proof}

%However we will see in Corollary \ref{cor:kron-dual} that for projective, rational, $\Q$-factorial T-varieties of complexity one which are contraction-free, the Kroenecker duality map \eqref{eq:kron-dual} is in fact an isomorphism. 

We now consider a more general complete, rational (not necessarily contraction-free) T-variety of complexity one. In this case, the Chow homology groups are not longer generated by divisors (see \cite{LLM} for an example). Nevertheless, in \cite[Theorem~4.1]{UN}, the author gives a nice representation in terms of generators and relations:
\begin{theorem}\label{th:chow-exact}
For a rational, complete T-variety $X$ of complexity one, there is for any $0 \leq k \leq \dim(X)$ an exact sequence 
\begin{align}\label{eq:chow-exact}
\bigoplus_{F \in V_{k+1}}M(F) \bigoplus_{\tau \in R_{k+1}}\left(M(\tau) \oplus \Z^P/\Z\right)\bigoplus_{\tau \in T_{k+1}}M(\tau) \longrightarrow \Z^{V_k} \oplus \Z^{R_k} \oplus \Z^{T_k} \longrightarrow A_k(X) \longrightarrow 0,
\end{align}
where the maps are given in the following way.

Let $F \in V_{k+1}$. Then we can choose a point $p \in p$ such that $Z_{p,F}$ corresponds to a $(k+1)$-dimensional invariant subvariety of a toric subvariety with character lattice $M(F)$. Then $m \in M(F)$ gets sent to 
\[
\sum_{\substack{G \in S_p(n-k) \\ F \prec G, \;\rec(G) \notin K}} \left\langle m,v_{F,G}\right\rangle Z_{p,G} + \sum_{\substack{G \in S_p(n-k) \\ F \prec G, \;\rec(G) \in K}} \left\langle m , v_{F,G}\right\rangle \frac{s_{\rec(G)}}{\mu(G)} W_{\rec(G)},
\]
where $v_{F,G}$ is defined in Notation \ref{notation}. 

Let $\tau \in R_{k+1}$. Then $B_{\tau}$ corresponds to a $(k+1)$-dimensional T-variety of complexity one. Then on the one hand side, a character $m \in M(\tau)$ is mapped to 
\[
\sum_{\substack{F \in V_k\\\rec(F) = \tau}} \mu(v_F)\left\langle m , v_F\right\rangle Z_{p,F} + \sum_{\substack{\sigma \in R_k\\ \tau \prec \sigma }}\left\langle m, v_{\tau,\sigma} \right\rangle B_{\sigma}
\]
where $v_{\tau,\sigma}$ is defined in Notation \ref{notation} and $v_F$ in Notation \ref{not:vertex}.
On the other hand, a generator of $\Z^p/\Z$ corresponds to a divisor $[p] -[ \infty]$. This is mapped to 
\begin{align}\label{eq:special}
\sum_{\substack{F \in S_p(n-k)\\ \rec(F) = \tau}} \mu(v_F) Z_{p,F} - \sum_{\substack{F \in S_{\infty}(n-k) \\ \rec(F) = \tau}} \mu(v_F) Z_{\infty,F}.
\end{align}

Finally, if $\tau \in T_{k+1}$, then every cone $\sigma \in \Sigma$ containing $\tau$ is also contracted. In particular, such a $(n-k)$-dimensional cone $\sigma$ lies in $T_k$. Then $m \in M(\tau)$ gets mapped to 
\[
\sum_{\substack{\sigma \in \Sigma(n-k) \\ \tau \prec \sigma}} \left\langle m , v_{\tau, \sigma}\right\rangle W_{\sigma}.
\]

\end{theorem}
\begin{proof}
This is \cite[Theorem~4.1]{UN}.
\end{proof}

\begin{rem}\label{rem:fin-pre}
\begin{enumerate}
\item[(i)] Note that the surjectivity of the map follows from Theorem~\ref{th:eff_cone}. The exactness of the sequence is then a consequence of the fact that the maps are exactly given
by choosing an invariant subvariety $Z$ of dimension $k + 1$, choosing a rational
function on $Z$ and taking its divisor. We refer for details to the proof in \emph{loc. cit.}
\item[(ii)]
The maps describe the relations between the invariant cycles in $A_k(X)$. For example, note that equation \eqref{eq:special} is saying that for $\tau \in R_{k+1}$, any two special fibers of $B_{\tau}$ are rationally equivalent. 
\item[(iii)] Since the set $P$ of points in $\mathbb{P}^1$ such that the polyhedral complex $\S_p$ is different from $\Sigma$ is finite, Theorem $\ref{th:chow-exact}$ gives a finite presentation of the Chow groups.
\end{enumerate}
\end{rem}

%We obtain the following corollary.
%\begin{cor}\label{cor:kron}
%Let $X$ be a complete $T$-variety (not necessarily of complexity one). Then the Kroenecker duality map 
%\[
%A^q(X) \longrightarrow \Hom(A_q(X),\Z)
%\]
%is an isomorphism. 
%\end{cor}

\section{Generalized Minkowski weights}\label{sec:gen-mink}
Let $X$ be a $(n+1)$-dimensional complete, rational T-variety of complexity one with divisorial fan $\S$ over $Y= \P^1$, recession fan $\Sigma$ in $N_{\Q}$ and subset $K \subseteq \Sigma$. Assume that $K$ is a proper subset of $\Sigma$, i.e. $K \neq \Sigma$. 

As before, for $p \in \P^1$ we write $\S_p$ for the corresponding complete polyhedral subdivision of $N_{\Q}$. Let $P$ be the set of points of $\P^1$ such that $\S_p \neq \Sigma$. $P$ is a finite set and we assume that $P$ has size at least two. 

For an integer $k \geq 0$ let $R_k, V_k, T_k$ be the sets defined in Section~\ref{sec:chow}. Recall that we have a Kroenecker duality map 
\[
A^k(X) \longrightarrow \Hom(A_k(X),\Z),
\]
which by Proposition~\ref{prop:kron-dual} is an isomorphism after tensoring with $\Q$ if $X$ is projective, $\Q$-factorial and contraction-free.

Then it follows from Theorem \ref{th:chow-exact} that to an element in $A^k(X)$ we can associate an integer-valued function on $R_k \bigcup V_k\bigcup T_k$ satisfying some properties, and that this association is bijective whenever the Kroenecker duality map is an isomorphism. We will describe such functions explicitly.
\begin{Def}\label{def:gen-mink}
A \emph{weight on $X$ of codimension $k$} is a map
\[
c \colon V_k\bigcup R_k \bigcup T_k \longrightarrow \Z.
\]
It is called a \emph{generalized Minkowski weight of codimension $k$} if it satisfies the following three conditions.
\begin{enumerate}
\item For $F \in V_{k+1}$ choose a point $p$ such that $Z_{p,F}$ corresponds to an invariant $(k+1)$-dimensional subvariety of a toric varitety with lattice of characters $M(F)$. Then for all $m \in M(F)$
\[
\sum_{\substack{G \in S_p(n-k) \\ F \prec G, \; \rec(G) \notin K}} \left\langle m, v_{F,G}\right\rangle c(G) + \sum_{\substack{G \in S_p(n-k) \\ F \prec G, \; \rec(G) \in K}} \left\langle m, v_{F,G}\right\rangle \frac{s_{\rec(G)}}{\mu(G)} c(\rec(G)) = 0.
\]
\item Let $\tau \in R_{k+1}$. Then for all $m \in M(\tau)$
\begin{align}\label{eq:cond_2.1}
& \sum_{\substack{F \in V_k\\ \rec(F) = \tau}} \mu(v_F)\left\langle m, v_F\right\rangle c(F) + \sum_{\substack{\sigma \in R_k \\ \tau \prec \sigma}}\left\langle m, v_{\tau,\sigma}\right\rangle c(\sigma) = 0. 
\end{align}
Assume $\infty \in P$. Then for any point $p\in P$ 
\begin{align}\label{eq:cond_2.2}
\sum_{\substack{F \in S_p(n-k)\\ \rec(F) = \tau}} \mu(v_F)c(F) - \sum_{\substack{F \in S_{\infty}(n-k) \\ \rec(F) = \tau}} \mu(v_F) c(F) = 0.
\end{align}

\item For all $\tau \in T_{k+1}$ and for all $m \in M(\tau)$ 
\[
\sum_{\substack{\sigma \in \Sigma(n-k) \\ \tau \prec \sigma}}\left\langle m, v_{\tau,\sigma}\right\rangle c(\sigma) = 0.
\]
\end{enumerate}
We denote by $W_k(X)$ the set of weights on $X$ of codimension $k$. It is an abelian group under the addition of functions. We denote by $M_k(X)$ the subgroup of generalized Minkowski weight of codimension $k$ and we refer to the linear conditions (1), (2) and (3) as \emph{balancing conditions}. 
\end{Def} 
\begin{rem}
The term \emph{balanced} comes from tropical geometry, where it plays a fundamental role. Indeed, tropical cycles are always \emph{balanced} and this condition is necessary for example in order to have a well-defined tropical intersection theory.
\end{rem}
\begin{Def}
We define a tropicalization map 
\[
\on{trop} \colon A^k(X) \longrightarrow W_k(X)
\]
by
\[
 z \longmapsto \on{trop}(z),
\]
where 
\[
\on{trop}(z)\colon V_k\bigcup R_k \bigcup T_k \longrightarrow \Z
\]
is given by 
\[
\gamma \longmapsto z \cap \left[V_{\gamma}\right].
\]
Here, $V_{\gamma}$ denotes the invariant subvariety corresponding to $\gamma$. 
\end{Def}
The following is our first main result. 
\begin{theorem}\label{thm:trop}
Let $X$ be a complete rational T-variety of complexity one. 
Then the tropicalization map
\[
\on{trop}\colon  A^k(X) \longrightarrow M_k(X)
\]
 is well-defined. If $X$ is projective, $\Q$-factorial and contraction-free, then the induced tropicalization map 
 \begin{equation}\label{eq:tropi}
\on{trop}\colon  A^k(X)_{\Q} \longrightarrow M_k(X)_{\Q},
\end{equation}
where $M_k(X)_{\Q}$ denotes $\Q$-valued generalized Minkowski weights, is a bijection.
\end{theorem}
\begin{proof}
By the above discussion, this follows directly from Proposition~\ref{prop:kron-dual} and Theorem~\ref{th:chow-exact}.
\end{proof}
We will now study the restriction of weights to invariant subvarieties.
\begin{Def}
Let $c \in W_k(X)$. 
\begin{itemize}
\item For $\tau \in \Sigma$ let $\iota \colon B_{\tau} \hookrightarrow \tilde{X}$ the inclusion of the corresponding subvariety. Recall that $\dim(B_{\tau}) =\codim(\tau) +1$. We define a weight $c^{\tau}$ on $B_{\tau}$ in the following way. Consider the sets
\begin{align*}
V_k(\tau) &= \left\{ F \in S_p(n-k), \; p \in P \text{ such that }\tau \prec \rec(F), \; \rec(F) \notin K\right\},\\
R_k(\tau) &= \left\{ \sigma \in \Sigma(n-k+1), \; \tau \prec \sigma, \; \sigma \notin K \right\},\\
T_k(\tau) &= \left\{ \sigma \in \Sigma(n-k), \; \tau \prec \sigma, \; \sigma \in K \right\}.
\end{align*}
Then 
\[
c^{\tau} \colon V_k(\tau) \bigcup R_k(\tau) \bigcup T_k(\tau) \longrightarrow \Z
\]
is given by 
\begin{align*}
\gamma \longmapsto c(\gamma).
\end{align*}
The weight $c^{\tau}$ can be thought of as the restriction of $c$ to the invariant subvariety $B_{\tau}$. 

\item On the other hand, let $p \in \P^1$ and $G \in S_p$ such that $Z_{p,G}$ is an invariant subvariety of a toric subvariety $X_{\Sigma(p,G)}$. Then if $\rec(G) \notin K$, we define $c^{p,G}$ as a function
\[
c^{p,G} \colon \Sigma(p,G) =\left\{ F \in S_p(n-k), \; G \prec F \right\} \longrightarrow \Z
\]
defined by
\[
 F \longmapsto c^{p,G}(F),
\]
where $c^{p,G}(F) = c(F)$ if $\rec(F) \notin K$. If $\rec(F) \in K$, then
by the proof of Theorem \ref{th:chow-exact}, one can write $Z_{p,F}$ in $A_k(X)$ as an integral linear combination 
\[
Z_{p,F} = \sum_{\gamma}a_{\gamma}V_{\gamma},
\]
the sum being over all $\gamma$ in $V_k \bigcup R_k \bigcup T_k$. Then 
we set
\[
c^{p,G}(F) = \sum_{\gamma}a_{\gamma}c(\gamma).
\]
Then $c^{p,G}$ can be thought of as the restriction of $c$ to the invariant subvariety $F_{p,G}$.
\end{itemize}
\end{Def}
\begin{lemma}
Let $c \in M_k(X)$ be a generalized Minkowski weight. 
\begin{enumerate}
\item For any $\tau \in\Sigma$ the restriction $c^{\tau}$ is a generalized Minkowski weight in $M_k(B_{\tau})$. 
\item For any point $p \in \P^1$ and $G \in S_p$ such that $Z_{p,G}$ is an invariant subvariety of a toric subvariety $X_{\Sigma(p,G)}$, and such that $\rec(G) \notin K$, the restriction $c^{p,G}$ is a Minkowski weight on the fan $\Sigma(p,G)$ in the sense of \cite{FS}.
\end{enumerate}
\end{lemma}
\begin{proof}
Let $c = \on{trop}(z)$ for some $z \in A^k(X)$. 

For $(1)$ we claim that
\[
\on{trop}(\iota^* z) = c^{\tau}.
\]
Indeed, it follows from the description of $B_{\tau}$ as a T-variety that the sequence 
\[
\Z^{V_k(\tau)} \oplus \Z^{R_k(\tau)} \oplus \Z^{T_k(\tau)} \longrightarrow A_k(B_{\tau}) \longrightarrow 0
\]
is exact. Hence, by construction, the domain of $\on{trop}(\iota^* z) \in M_k(B_{\tau})$ is indeed $V_k(\tau) \bigcup R_k(\tau) \bigcup T_k(\tau)$. 

Let $\gamma \in V_k(\tau) \bigcup R_k(\tau) \bigcup T_k(\tau)$. We have to check that 
\[
\on{trop}(\iota^*Z)(\gamma) = c^{\tau}(\gamma).
\]
But this follows from the definition of $c^{\tau}$ and the projection formula. 

For $(2)$ we proceed similarly. Let $j \colon X_{\Sigma(p,G)} \hookrightarrow \tilde{X}$ be the inclusion of the toric subvariety. Then $\on{trop}(j^*z)$ is a Minkowski weight on the fan $\Sigma(p,G)$ in the sense of \cite{FS}. We claim that 
\[
\on{trop}(j^*Z) = c^{p,G}.
\]
Indeed, we have to check that for $F \in S_p(n-k)$ with $G \prec F$, we have that 
\[
\on{trop}(j^*Z)(F) = c^{p,G}(F).
\]
Again, this follows from the definition of $c^{p,G}$ and the projection formula. 

\end{proof}
\section{Intersecting with Cartier divisors}\label{sec:intersect}
Let $X$ be a $(n+1)$-dimensional complete, rational T-variety of complexity one with divisorial fan $\S$ over $Y= \P^1$ and recession fan $\Sigma$ in $N_{\Q}$. In this section we assume that $X$ is projective, $\Q$-factorial and contraction free.

As before, for $p \in \P^1$ we write $\S_p$ for the corresponding complete polyhedral subdivision of $N_{\Q}$. Let $P$ be the set of points of $\P^1$ such that $\S_p \neq \Sigma$. $P$ is a finite set and we assume that $P$ has size at least two. 

The goal of this section is to define a pairing
\[
\CaSF(\S) \times M_k(X) \longrightarrow M_{k+1}(X), \quad (h, c) \longmapsto h \cdot c
\] 
and to show that it is compatible with the isomorphism \eqref{eq:tropi} from Theorem \ref{thm:trop}.

We can think of this pairing as a rational analogue of the so called \emph{corner locus} of a tropical cycle in the sense of \cite[Section~6]{AR}.

In order to define this pairing we use the combinatorial data of $\S$ and $\Sigma$. Note that since we are assuming that $X$ is contraction free, we have that $T_k(X) = \emptyset$ for any integer $k \geq 0$. 
\begin{Def}
Let $h = \left(h_p\right)_{p\in \P^1} \in \CaSF(\S)$ be a Cartier divisorial support function on $\S$ with recession function $\rec(h) = \overline{h}$.  And let $c \in M_k(X)$ a generalized Minkowski weight of codimension $k$. First we define a paring 
\begin{align}\label{eq:pairing}
\CaSF(S) \times M_k(X) \longrightarrow W_{k+1}(X), \quad (h,c) \longmapsto h\cdot c,
\end{align}
where $h \cdot c \colon V_{k+1} \bigcup R_{k+1} \to \Z$ is defined in the following way.
\begin{itemize}
\item Let $F \in V_{k+1}$. Choose a point $p \in P$ such that $Z_{p,F}$ corresponds to an invariant subvariety of a toric variety. Then we set

\begin{align*}
(h \cdot c)(F) = \sum_{\substack{G \in S_p(n-k) \\ F \prec G}} -h_{p,F}(v_{F,G})c(G),
\end{align*}
where $h_{p,F}$ is the function defined in Definition \ref{def:restriction-div}.
\item Let $\tau \in R_{k+1}$. For any $\sigma \in R_k$ containing $\tau$ and for any $F \in V_k$ with $\rec(F) = \tau$ we consider lifts of $v_{\tau,\sigma}$ and $v_F$ in $N_{\R}$. We denote these lifts by the same symbols, in order to avoid the notation becoming even heavier. Then we set
\begin{align*}
 (h \cdot c)(\tau) = & \sum_{\substack{F \in V_k \\ \rec(F) = \tau}}-\mu(v_F)h(v_F)c(F) + \sum_{\substack{\sigma \in R_k \\ \tau \prec \sigma}}-\overline{h}_{\sigma}(v_{\tau,\sigma}) c(\sigma)  \\
& + \overline{h}_{\tau}\left(\sum_{\substack{F \in V_k \\ \rec(F) = \tau}} \mu(v_F)v_Fc(F) + \sum_{\substack{\sigma \in R_k \\ \tau \prec \sigma}}v_{\tau,\sigma}c(\sigma)\right),
\end{align*} 
where for $\gamma \in \Sigma$, $\overline{h}_{\gamma}$ denotes a linear function on $N_{\R}$ which agrees with $\overline{h}$ on $\gamma$.
% induced by $\overline{h}$  (take $\overline{h}-\overline{h}_{\gamma}$ where $\overline{h}_{\gamma}$ is a linear function which restricted to $\gamma$ coincides with $\overline{h}$. 

\end{itemize}
\end{Def}
Note that since $c$ is a generalized Minkowski weight, the pairing \eqref{eq:pairing} is well defined, i.e.~it does not depend on the choice of lifts.

We will now proceed to show that $h \cdot c \in W_{k+1}(X)$ satisfies the balancing conditions (1),(2) and (3) from Definition \ref{def:gen-mink}, and is thus a generalized Minkowski weight. This will follow from the following three important lemmas. 
\begin{lemma}\label{lem:restriction}
Let $c \in M_k(X)$ and $h \in \CaSF(S)$. Then the following holds true. 
\begin{enumerate}
\item\label{it1} For all $\tau \in \Sigma$ we have
\[
h(\tau) \cdot c^{\tau} = (h \cdot c)^{\tau}
\]
as weights in $W_{k+1}(B_{\tau})$, where $h(\tau)$ denotes the function in Definition \ref{def:restriction-div}.
\item\label{it2} Let $p \in P$ and $G \in S_p$ such that $Z_{p,G}$ is an invariant subvariety of a toric subvariety $X_{\Sigma(p,G)}$. Then 
\[
h_{p,G} \cdot c^{p,G} = (h \cdot c)^{p,G}
\]
as Minkowski weights in the fan $\Sigma(p,G)$. Here, the intersection pairing $h_{p,G} \cdot c^{p,G}$ is the classical one for fans, as defined in \cite[Definition 3.4]{AR}. 
\end{enumerate}

\end{lemma}
\begin{proof}
Part \ref{it2} follows from toric geometry and the compatibility of the intersection pairing with the pullback (see e.g. \cite[Chapter~17]{fulint}). 

%A proof can be also found in a more general setting in \cite[Proposition~3.17]{GR}).
For part \ref{it1} we first assume that $\overline{h}|_{\tau} = 0$. We have to show that for all $\omega \in R_{k+1}(\tau)\bigcup V_{k+1}(\tau)$ we have that 
\[
\left(h(\tau)\cdot c^{\tau}\right)(\omega) = (h \cdot c)^{\tau}(\omega).
\] 
\begin{itemize}
\item Let $F \in V_{k+1}(\tau)$ and choose a point $p \in P$ such that $Z_{p,F}$ corresponds to an invariant subvariety of a toric variety.Then 
\begin{align*}
(h \cdot c)^{\tau}(F) &=  (h \cdot c)(F) \\
& = \sum_{\substack{G \in S_p(n-k) \\ F \prec G}} -h_{p,F}(v_{F,G})c(G) \\
&= \sum_{\substack{G \in S_p(n-k)\\ F \prec G}} -h(\tau)_{p,F}(v_{F,G})c^{\tau}(G) \\
&= \left(h(\tau) \cdot c^{\tau}\right)(F)
\end{align*}

\item Let $\gamma \in R_{k+1}(\tau)$. Then 
\begin{align*}
& (h \cdot c)^{\tau}(\gamma) =  (h \cdot c)(\gamma) \\ = &\sum_{\substack{F \in V_k \\ \rec(F) = \gamma}}-\mu(v_F)h(v_F)c(F) + \sum_{\substack{\sigma \in R_k \\ \gamma \prec \sigma}}-\overline{h}_{\sigma}(v_{\gamma,\sigma}) c(\sigma)  \\
&  + \overline{h}_{\gamma}\left(\sum_{\substack{F \in V_k \\ \rec(F) = \gamma}} \mu(v_F)v_Fc(F) + \sum_{\substack{\sigma \in R_k \\ \gamma \prec \sigma}}v_{\gamma,\sigma}c(\sigma)\right) \\
=& \sum_{\substack{F \in V_k(\tau) \\ \rec(F) = \gamma}}-\mu(v_F)h(\tau)(v_F)c^{\tau}(F) + \sum_{\substack{\sigma \in R_k(\tau) \\ \gamma \prec \sigma}}-\overline{h(\tau)}_{\sigma}(v_{\gamma,\sigma}) c^{\tau}(\sigma)  \\
& + \overline{h(\tau)}_{\gamma}\left(\sum_{\substack{F \in V_k(\tau) \\ \rec(F) = \gamma}} \mu(v_F)v_Fc^{\tau}(F) + \sum_{\substack{\sigma \in R_k(\tau) \\ \gamma \prec \sigma}}v_{\gamma,\sigma}c^{\tau}(\sigma)\right)\\
=& \left(h(\tau)\cdot c^{\tau}\right)(\gamma).
\end{align*}
If $\overline{h}|_{\tau} = m_{\tau} \neq 0$ then we consider the divisorial support function $h - m_{\tau}$. Then by linearity of $m_{\tau}$ we have that for all $\omega\in R_{k+1}(\tau)\bigcup V_{k+1}(\tau)$
\[
\left(h \cdot c \right)^{\tau}(\omega) = \left((h-m_{\tau})\cdot c\right)^{\tau}(\omega) = \left((h-m_{\tau})(\tau)\cdot c^{\tau}\right)(\omega) = \left(h(\tau)\cdot c^{\tau}\right)(\omega).
\]
This finishes the proof of the lemma.

\end{itemize}
\end{proof}

\begin{lemma}\label{lem:prod}
Let $h, h' \in \CaSF(\S)$. Then for any $c \in M_k(X)$ the following holds true. \begin{enumerate}
\item\label{prod1} $(h + h')\cdot c = h\cdot c + h' \cdot c$,
\item\label{prod2} $ h\cdot (h' \cdot c) = h' \cdot (h \cdot c)$.
\end{enumerate}
\end{lemma}
\begin{proof}
Part \ref{prod1} follows easily from the definition of the pairing.
We now prove part~\ref{prod2}. As before, we have to show that 
\[
\left(h \cdot (h' \cdot c)\right)(\omega) = \left(h' \cdot (h \cdot c)\right)(\omega)
\]
for any $\omega \in V_{k+2} \bigcup R_{k+2}$. 
\begin{itemize}
\item Let $F \in V_{k+2}$ and choose a point $p \in P$ such that $Z_{p,F}$ corresponds to an invariant subvariety of a toric variety. 
%We first assume that $h_{p, F}= h'_{p,F} = 0$. 
We compute 
\begin{align*}
& \left(h \cdot (h' \cdot c)\right)(F) = \sum_{\substack{G \in S_p(n-k-1)\\ F \prec G}} -h_{p,F}(v_{F,G})(h' \cdot c)(G) \\
=& \sum_{\substack{G \in S_p(n-k-1)\\ F \prec G}} -h_{p,F}(v_{F,G}) \left(\sum_{\substack{H \in S_p(n-k) \\ G \prec H}}- h'_{p,G}(v_{G,H})c(H)\right).
\end{align*}
Now, for $H \in S_p(n-k)$, $G \in S_p(n-k-1)$ with $F \prec G \prec H$ we let $G' \in S_p(n-k-1)$ the unique polyhedron such that $F \prec G'\prec H$. Let $\alpha_{F,H}$ and $\beta_{F,H}$ be elements such that 
\begin{align}\label{eq:v1}
v_{G,H} = \alpha_{F,H}v_{F,G} +  \beta_{F,H}v_{F,G'}
\end{align}
and
\begin{align}\label{eq:v2}
v_{G',H} = \alpha_{F,H}v_{F,G'} +  \beta_{F,H}v_{F,G}.
\end{align}

Using \eqref{eq:v1} we obtain 
\begin{align*}
\left(h \cdot (h' \cdot c)\right)(F)   = &\sum_{\substack{G \in S_p(n-k-1)\\ F \prec G}} -h_{p,F}(v_{F,G}) \left(\sum_{\substack{H \in S_p(n-k) \\ G \prec H}}- h'_{p,G}(v_{G,H})c(H)\right) \\
= &   \sum_{\substack{G \in S_p(n-k-1)\\ F \prec G}} \sum_{\substack{H \in S_p(n-k)\\ G \prec H}} c(H)h_{p,F}(v_{F,G})h'_{p,G}(v_{G,H}) \\
=& \sum_{\substack{G \in S_p(n-k-1)\\ F \prec G}} \sum_{\substack{H \in S_p(n-k)\\ G \prec H}} c(H)\alpha_{F,H}h_{p,F}(v_{F,G})h'_{p,F}(v_{F,G}) \\
& + \sum_{\substack{G \in S_p(n-k-1)\\ F \prec G}} \sum_{\substack{H \in S_p(n-k)\\ G \prec H}} c(H)\beta_{F,H}h_{p,F}(v_{F,G})h'_{p,F}(v_{F,G'})\\
=& \sum_{\substack{H \in S_p(n-k)\\F \prec H}}  c(H) \alpha_{F,H}\left(\sum_{\substack{G \in S_p(n-k-1)\\ F \prec G \prec H}}h_{p,F}(v_{F,G})h'_{p,F}(v_{F,G})\right) \\
&+ c(H) \beta_{F,H} \left(\sum_{\substack{G \in S_p(n-k-1)\\ F \prec G \prec H}}h_{p,F}(v_{F,G})h'_{p,F}(v_{F,G'})\right).
\end{align*}
Next, we interchange the roles of $G$ and $G'$ in the inner sum of the second summand of the last expression and revert the argument. Using \eqref{eq:v2} we get
\begin{align*}
\left(h \cdot (h' \cdot c)\right)(F) 
=& \sum_{\substack{H \in S_p(n-k)\\F \prec H}}  c(H) \alpha_{F,H}\left(\sum_{\substack{G \in S_p(n-k-1)\\ F \prec G \prec H}}h_{p,F}(v_{F,G})h'_{p,F}(v_{F,G})\right) \\
&+ c(H) \beta_{F,H} \left(\sum_{\substack{G \in S_p(n-k-1)\\ F \prec G \prec H}}h_{p,F}(v_{F,G'})h'_{p,F}(v_{F,G})\right)\\
=& \sum_{\substack{G \in S_p(n-k-1)\\ F \prec G}} \sum_{\substack{H \in S_p(n-k)\\ G \prec H}} c(H)h'_{p,F}(v_{F,G})h_{p,G}\left(\alpha_{F,H}v_{F,G} + \beta_{F,H}v_{F,G'}\right) \\
=& \sum_{\substack{G \in S_p(n-k-1)\\ F \prec G}} \sum_{\substack{H \in S_p(n-k)\\ G \prec H}} c(H)h'_{p,F}(v_{F,G})h_{p,G}(v_{G,H})\\
= &\; \left(h' \cdot (h \cdot c)\right)(F),
\end{align*}
which is what we wanted to show. 
\item Let $\tau \in R_{k+2}$. By the restriction Lemma \ref{lem:restriction} we may assume $n=2$, $k=1$. Then
\begin{align*}
\left(h \cdot(h'\cdot c)\right)(\{0\}) = \sum_{F \in V_2}-\mu(v_F)h(v_F)(h'\cdot c)(F) + \sum_{\sigma \in R_2}-\overline{h}_{\sigma}(v_{\sigma})(h' \cdot c)(\sigma).
\end{align*}
Let 
\[
A\coloneqq \sum_{F \in V_2}-\mu(v_F)h(v_F)(h'\cdot c)(F) \; \;  \text{ and } \; \;
B\coloneqq \sum_{\sigma \in R_2}-\overline{h}_{\sigma}(v_{\sigma})(h' \cdot c)(\sigma).
\]
By definition of the pairing, we have 
\begin{align*}
A = \sum_{F \in V_2}-\mu(v_F)h(v_F)\left(\sum_{\substack{G \in S_p(1) \\ F \prec G}}-h'_{p,F}(v_{F,G})c(G) \right).
\end{align*}
We can write the above sum as a sum over all one-dimensional faces $G$ in $S_p(1)$ by distinguishing whether $\rec(G) = \{0\}$ or $\rec(G) \in \Sigma(1)$. Note however that in the former case, the terms in the sum cancel each other out since $v_{v_1,G} = - v_{v_2,G}$, where $v_1, v_2$ are the vertices corresponding to a bounded edge $G$. Hence, we can write the sum in $A$ as a sum over all unbounded edges $G$ in $S_p(1)$. For such a $G$ we write $v_G$ for its unique vertex. Then we get
\begin{align}\label{eq:A}
A =   \sum_{\substack{G \in S_p(1) \\ \{0\} \neq \rec(G)}} \mu(v_G)h(\rec(G))(v_G)h'_{p,v_G}(v_{v_G,G})c(G). 
\end{align}

Now, for $B$, by Lemma \ref{lem:restriction} and the definition of the pairing, we have
\begin{align*}
B =&\sum_{\sigma \in R_2}-\overline{h}_{\sigma}(v_{\sigma})(h' \cdot c)(\sigma)\\
= &\sum_{\sigma \in R_2}-\overline{h}_{\sigma}(v_{\sigma})(h'(\sigma) \cdot c^{\sigma})(\{0\}) \\
= &\sum_{\sigma \in R_2}\overline{h}_{\sigma}(v_{\sigma}) \left(\sum_{\substack{G \in S_p(1) \\ \sigma = \rec(G)}}-\mu(v_G)h'(\sigma)(v_G)c^{\sigma}(G)  +\sum_{\substack{\gamma \in \Sigma(2)\\ \sigma \prec \gamma}} -\overline{h'(\sigma)}_{\gamma}(v_{\sigma,\gamma})c^{\sigma}(\gamma)\right) \\
=& \sum_{\sigma \in R_2}\overline{h}_{\sigma}(v_{\sigma}) \left(\sum_{\substack{G \in S_p(1) \\ \sigma = \rec(G)}}-\mu(v_G)h'(\sigma)(v_G)c(G)  +\sum_{\substack{\gamma \in \Sigma(2)\\ \sigma \prec \gamma}} -\overline{h'}_{\gamma}(v_{\sigma,\gamma})c(\gamma)\right)
\end{align*}

Now, note that for $\sigma \in R_2$ and for any $G \in S_p(1)$ such that $\rec(G) = \sigma$, we have $\overline{h}_{\sigma}(v_{\sigma}) = h_{p,v_G}(v_{v_G,G})-h(\sigma)(v_G)$. Also, for $G, G' \in S_p(1)$ such that $\rec(G) + \rec(G') \in \Sigma(2)$ we have
\[
\overline{h'}_{\rec(G) + \rec(G')}\left(v_{\rec(G')}\right) = h'_{p,v_G}(v_{v_G,G'}) - h'(\rec(G'))(v_{G'})
\]
Then we can write the above sum as a sum over pairs of one-dimensional faces in the following way. 
\begin{align}\label{eq:B}
B =& \sum_{\substack{G,G' \in S_p(1) \\ \rec(G) + \rec(G') \in \Sigma(2)}} \left(h_{p,v_G}(v_{v_G,G})-h(\rec(G))(v_G)\right) \Big[  \mu(v_G)h'_{p,v_G}(v_{v_G,G})c(G)  \nonumber \\ & + \left(h'_{p, v_{G'}}(v_{v_G',G'}) - h'(\rec(G')(v_{G'})\right)c\left(\rec(G) + \rec(G')\right) \Big] \nonumber \\
=& \sum_{\substack{G \in S_p(1) \\ \{0\} \neq \rec(G)}}  -\mu(v_G)c(G)h(\rec(G))(v_G)h'_{p,v_G}(v_{v_G,G}) + Z(h,h'),
\end{align}
where $Z$ is symmetric in $h,h'$, i,e. $Z(h, h') = Z(h', h)$. 

From \eqref{eq:A} and \eqref{eq:B} we get 
\[
\left(h \cdot (h' \cdot c)\right)(\{0\}) = A+B = Z(h,h') = Z(h', h).
\]
We can now redo the above argument for $\left(h'\cdot (h \cdot c)\right)(\{0\})$ and we
obtain
\[
\left(h'\cdot (h \cdot c)\right)(\{0\}) = Z(h', h)= Z(h,h') = \left(h \cdot (h' \cdot c)\right)(\{0\}),
\]
which is what we wanted to show.

\end{itemize}
\end{proof}

\begin{lemma}\label{lem:char-min}
Let $c \in W_k(X)$ be a weight of codimension $k$. Then $c$ is a generalized Minkowski weight if and only of for all principal $h \in CaSF(\S)$ we have that $h \cdot c = 0$.
\end{lemma}
\begin{proof}
Assume that $c$ is a generalized Minkowski weight and let $h \in CaSF(\S)$ be principal. We have to show that $h \cdot c = 0$. By the first part of Lemma \ref{lem:prod} it suffices to show this for the case $h= SF(u)$ for $u \in M$, and for the case $h = SF\left([p]-[\infty]\right)$, $p \in P$. 

The case $h= SF(u)$ follows easily by linearity and the definition of the pairing. 

Assume now that $h = SF\left([p]-[\infty]\right)$.

Let $\tau \in R_{k+1}$. Then we have
\begin{align}\label{eq:con2}
h \cdot c(\tau) = \sum_{\substack{F \in S_p(n-k)\\ \rec(F) = \tau}}-\mu(v_F)c(F) + \sum_{\substack{F \in S_{\infty}(n-k)\\ \rec(F) = \tau}}\mu(v_F)c(F) = 0
\end{align}
by condition (2) of a generalized Minkowski weight. 

On the other hand, for $F \in V_{k+1}$, let $\tilde{p} \in P$ be a point such that $Z_{\tilde{p},F}$ corresponds to an invariant subvariety of a toric variety. We may assume $p = \tilde{p}$. Then $h \cdot c(F) = 0$ by the definition of the pairing since $h_{p,F}= 0$ in this case.

Conversely, assume that $h \cdot c =0$ for all principal $h$. We have to show that $c$ satisfies the balancing conditions of Definition \ref{def:gen-mink}.
\begin{itemize}
\item Let $F \in V_{k+1}$ and choose a point $p \in P$ such that $Z_{p,F}$ corresponds to an invariant subvariety of a toric variety with lattice of characters $M(F)= N(F)^{\vee}$. Choose a representative $\tilde{m},$ of $m$ in $M$. Let $h = \SF(\tilde m) + \SF( [p] -[\infty])$. We compute 
\begin{align*}
0 &= -\left(h \cdot c\right)^{p,F}(\{0\}) =h_{p,F}\cdot c^{p,F}(\{0\}) \\
&= \sum_{\alpha \in \Sigma(p,F)(1)}h_{p,F}(v_{\alpha}) c^{p,F}(\alpha) \\
&= \sum_{\substack{G \in S_p(n-k) \\ F \prec G}} \left\langle m, v_{F,G}\right\rangle c(G).
\end{align*}
Hence, $c$ satisfies condition $(1)$ of a generalized Minkowski weight. 
\item Let $\tau \in R_{k+1}$ and let $m$ be any element in $M(\tau)$. Choose any representative $\tilde{m} \in M$ of $m$. Consider the principal divisor $h = \SF(\tilde{m})$. Note that we have $h|_{\tau}= 0$. We compute 
\begin{align*}
 0 & = -h \cdot c(\tau) =  \sum_{\substack{F \in V_k \\ \rec(F) = \tau}}\mu(v_F)h(v_F)c(F) + \sum_{\substack{\sigma \in R_k \\ \tau \prec \sigma}}\overline{h}_{\sigma}(v_{\tau,\sigma}) c(\sigma) \\
& = \sum_{\substack{F \in V_k \\ \rec(F) = \tau}}\mu(v_F)\langle m, v_F \rangle c(F) + \sum_{\substack{\sigma \in R_k \\ \tau \prec \sigma}}\langle m, v_{\tau,\sigma}\rangle c(\sigma).
\end{align*}

Moreover, for $h = \SF([p]-[\infty])$, then as in \eqref{eq:con2} we have 
\[
0 = -h \cdot c(\tau) =  \sum_{\substack{F \in S_p(n-k)\\ \rec(F) = \tau}}\mu(v_F)c(F) - \sum_{\substack{F \in S_{\infty}(n-k)\\ \rec(F) = \tau}}\mu(v_F)c(F).
\]
This shows that $c$ also satisfies the balancing condition (2). 
Hence $c$ is a generalized Minkowski weight. 
\end{itemize}
\end{proof}

\begin{prop}\label{prop:mink}
Let $c \in M_k(X)$ and $h \in \SF(\S)$. Then $h \cdot c \in M_{k+1}(X)$. It follows that the pairing \eqref{eq:pairing} induces a pairing 
\[
\CaSF(\S) \times M_k(X) \longrightarrow M_{k+1}(X), \; (h, c) \longmapsto h \cdot c.
\]
\end{prop}
\begin{proof}
Let $h' \in \SF(S)$ be principal. By Lemma~\ref{lem:char-min} it suffices to show that $h' \cdot (h \cdot c)= 0$. By Lemma ~\ref{lem:prod} we have 
\[
h' \cdot (h \cdot c) = h \cdot (h' \cdot c) = 0,
\]
concluding the proof.
\end{proof}

The next result states that the pairing in Proposition \ref{prop:mink} is compatible with the tropicalization isomorphism \eqref{eq:tropi} from Theorem \ref{thm:trop}.
\begin{theorem}\label{th:comp}(Compatibility with tropicalization)
As before we assume that $X$ is a projective, $\Q$-factorial, rational, contraction free T-variety of complexity one. Let $h \in \CaSF(\S)$ and $D_h$ its associated Cartier divisor. We set $[D_h] \in A^1(X)_{\Q}$ for its corresponding cohomology class. Furthermore, for $z \in A^k(X)_{\Q}$ we set $c = \on{trop}(z) \in M_k(X)_{\Q}$, where $\on{trop}$ is the isomorphism \eqref{eq:tropi} from Theorem~\ref{thm:trop}. Then 
\[
\on{trop}\left([D_h] \cup z \right) = h \cdot c
\]
as generalized Minkowski weights in $M_{k+1}(X)_{\Q}$.

\end{theorem}
\begin{proof}
Let $\omega \in R_{k+1}\bigcup V_{k+1}$. We have to show that 
\[
\trop\left([D_h] \cup z \right)(\omega) = h \cdot c (\omega).
\]
\begin{itemize}
\item Let $F \in V_{k+1}$. Choose a point $p \in P$ such that $Z_{p,F}$ corresponds to an invariant subvariety of a toric variety. Recall that $h_{p,F}$ denotes the (toric) virtual support function on the fan $\Sigma(p,F)$ corresponding to the pullback of $D_h$ to $X_{\Sigma(p,F)}$. Then 
\[
[D_h] \cap [Z_{p,F}] = \sum_{\substack{ G \in S_p(n-k) \\ F \prec G}}h_{p,F}(v_{F,G})[Z_{p,G}].
\]
Hence 
\begin{align*}
\trop\left([D_h] \cup z\right)(F) &= \deg\left(\left([D_h]\cup z \right) \cap [Z_{p,F}]\right) \\
& =\sum_{\substack{ G \in S_p(n-k) \\ F \prec G}}h_{p,F}(v_{F,G}) \deg\left(z \cap [Z_{p,G}]\right).
\end{align*}
Note that by definition we have $\deg\left(z \cap [Z_{p,G}]\right) = c^{p,F}(G)$. Then, using Lemma \ref{lem:restriction}, we obtain 
\begin{align*}
\trop\left([D_h] \cup z \right)(F) =& \sum_{\substack{ G \in S_p(n-k) \\ F \prec G}}h_{p,F}(v_{F,G}) \deg\left(Z \cap [Z_{p,G}]\right) \\
=& \sum_{\substack{ G \in S_p(n-k) \\ F \prec G}}h_{p,F}(v_{F,G}) c^{p,F}(G) \\
=& h_{p,F} \cdot c^{p,F}(F) \\
=& \left(h \cdot c\right)^{p,F}(F) \\
=& (h \cdot c)(F).
\end{align*} 
\item Let $\tau \in R_{k+1}$. By definition, we have 
\[
\trop\left([D_h] \cup z \right)(\tau) = \deg\left(\left([D_h] \cup z \right) \cap [B_{\tau}]\right).
\]
The restriction $h(\tau) \in \CaSF(\S(\tau))$ is the Cartier support function corresponding to $[D_h]\cap [B_{\tau}]$. Then 
\[
[D_h] \cap [B_{\tau}] = \left[D_{h(\tau)}\right] = - \sum_{\sigma \in R_k(\tau)}\overline{h(\tau)}_{\sigma}(v_{\tau,\sigma})[B_{\sigma}] - \sum_{\substack{F \in V_k(\tau)\\ \rec F = \tau}}\mu(v_F)h(\tau)(v_F)[Z_{p,F}].
\]
Hence the weight of $\trop\left([D_h] \cup z \right)$ at $\tau$ is 
\begin{align*}
\deg\left(\left([D_h] \cup z\right) \cap [B_{\tau}]\right) = &- \sum_{\sigma \in R_k(\tau)}\overline{h(\tau)}_{\sigma}(v_{\tau, \sigma}) \deg \left(z \cap [B_{\sigma}]\right) \\ &- \sum_{\substack{F \in V_k(\tau)\\ \rec F = \tau}}\mu(v_F)h(\tau)(v_F)\deg\left(z \cap [Z_{p,F}]\right).
\end{align*}
On the other hand we have 
\[
\deg \left(z \cap [B_{\sigma}]\right) = c^{\tau}(\sigma) \quad \text{and} \quad \deg\left(z \cap [Z_{p,F}]\right)=c^{\tau}(F).
\]
 Hence, using Lemma \ref{lem:restriction} we get
\begin{align*}
\trop\left([D_h] \cup c \right)(\tau) =&  - \sum_{\sigma \in R_k(\tau)}\overline{h(\tau)}_{\sigma}(v_{\tau, \sigma}) \deg \left(z \cap [B_{\sigma}]\right) \\ &- \sum_{\substack{F \in V_k(\tau)\\ \rec F = \tau}}\mu(v_F)h(\tau)(v_F)\deg\left(z \cap [Z_{p,F}]\right)\\
=&  - \sum_{\sigma \in R_k(\tau)}\overline{h(\tau)}_{\sigma}(v_{\tau, \sigma}) c^{\tau}(\sigma) - \sum_{\substack{F \in V_k(\tau)\\ \rec F = \tau}}\mu(v_F)h(\tau)(v_F)c^{\tau}(F) \\
=& \left(h(\tau) \cdot c^{\tau}\right)(\tau) \\
=& \left(h \cdot c\right)^{\tau}(\tau)\\
=& \left(h \cdot c\right)(\tau).
\end{align*}

\end{itemize}
\end{proof}
\section{The measure associated to an invariant Cartier divisor and top intersection numbers}\label{sec:measure}
Let $X$ be a $(n+1)$-dimensional complete, rational T-variety of complexity one with divisorial fan $\S$ over $Y= \P^1$ and recession fan $\Sigma$ in $N_{\Q}$. As in the last section, we assume that $X$ is projective, $\Q$-factorial and contraction free.

As before, for $p \in \P^1$ we write $\S_p$ for the corresponding complete polyhedral subdivision of $N_{\Q}$ and $P$ for the finite set of points of $\P^1$ such that $\S_p \neq \Sigma$. $P$, which is assumed to have size at least two. 

The goal of this section is to define a (discrete) Measure $\mu_h$ associated to a Cartier divisorial Support function $h \in \CaSF(\S)$ based on the combinatorial intersection pairing \eqref{eq:pairing} and to compute the top intersection number $D_h^{n+1}$ of the corresponding invariant divisor $D_h$ in terms of this measure. 

let $[X] \in A_{n+1}(X)$ denote the top Chow homology  class. By Poincaré duality and Theorem \ref{thm:trop}, we may consider the associated Minkowski weight $c_{[X]} \in M_0(X)$. This is given by 
\[
c_{[X]}\colon V_0 \longrightarrow \Z, \quad F \longmapsto 1.
\]
\begin{Def}\label{def:measure}
Let $h \in \CaSF(\S)$. 
The \emph{Total space} $\on{Tot}(\S)$ of the divisorial fan $\S$ is the real vector space 
\[
\on{Tot}(\S) \coloneqq \bigoplus_{p \in P}N_{\R} \bigoplus N_{\R}.
\]

Then to $h$ we associate a discrete measure $\mu_h$ on $\on{Tot}(\S)$ given in the following way. Consider the codimension $n$ generalized Minkowski weight defined inductively by 
\[
h^n \cdot c_{[X]} =  h \cdot \left(h^{n-1} \cdot c_{[X]}\right).
\]
Then $\mu_h$ is the measure supported on the set
\[
V_n \bigcup R_n,
\]
which is in bijection with the set $\bigsqcup_{p \in P}\S_p(0) \sqcup \bigsqcup_{\tau \in \Sigma(1)}v_{\tau}$.

The weights on $V_n$ are given by 
\[
F \longmapsto \mu(v_F)\left(h^n \cdot c_{[X]}(F)\right);
\]
and the weights on $R_n$ are given by 
\[
\tau \longmapsto h^n \cdot c_{[X]}(\tau).
\]
Note that the measure $\mu_h$ is positive if $D_h$ is nef. 
\end{Def}
We are now able to compute top intersection numbers of invariant Cartier divisors on $X$ in terms of their associated measures. 
\begin{cor}\label{cor:top-int}
Let $D_h$ be a T-invariant Cartier divisor on $X$ associated to $h \in \CaSF(\S)$ and let $\rec(h) = \overline{h}$ be its recession function.
\begin{enumerate}
\item The top intersection number $D_h^{n+1}$ can be computed inductively by 
\[
h^{n+1} \cdot c_{[X]}(\{0\}) = h^n\cdot \left(h \cdot c_{[X]}\right)(\{0\}).
\]
\item
Moreover, it can be computed as an integral with respect to the measure $\mu_h$ 
\[
D_h^{n+1}= \int_{\on{Tot}(\S)}-h\mu_h = \sum_{p \in P}\sum_{F \in \S_p(0)}-h(v_F)\mu_h(F)+ \sum_{\tau \in \Sigma(1)}-\overline{h}(v_{\tau})\mu_h(\tau).
\]
\end{enumerate}
\end{cor}
\begin{proof}
The first part follows from Theorem \ref{th:comp} and the second from the definition of the measure $\mu_h$  and the intersection pairing.
\end{proof}
\begin{exa}\label{exa:top}
Consider the Hirzebruch surface $\H_2= \on{Proj}\left(\O_{\P^1} \oplus \O_{\P^1}(2)\right)$. It is a toric variety whose corresponding fan is given in the following Figure \ref{fig:fan-h2} 
\begin{figure}[H]
\begin{center}
\begin{tikzpicture}[scale = 1.3]
		\draw  (-1.5,0) -- (1.5,0);
		\draw (0,-1.5) -- (0,1.5);
		\draw (0,0) -- (-1,2) node[above]{$(-1,2)$};
	\end{tikzpicture}\caption{Fan for $\H_2$}\label{fig:fan-h2}
	\end{center}
\end{figure}
We can perform a toric downgrade and consider $\H_2$ as a $T$-surface, where $T$ is a one-dimensional torus, as is done in \cite[Example 2.9]{PS}. Here one considers the embedding $T \hookrightarrow T_{\H_2}$ corresponding to the exact sequence
\[
0 \longrightarrow N \xrightarrow{\begin{pmatrix}
1\\ 0
\end{pmatrix}} N \oplus \Z 
\xrightarrow{\begin{pmatrix}
0 & 1 
\end{pmatrix}}\Z \longrightarrow 0.
\]
 This yields the divisorial fan $\S= \left(\S_p\right)_p$ with recession fan $\Sigma$ depicted in the following Figure~\ref{fig:poly-fan-h2}.

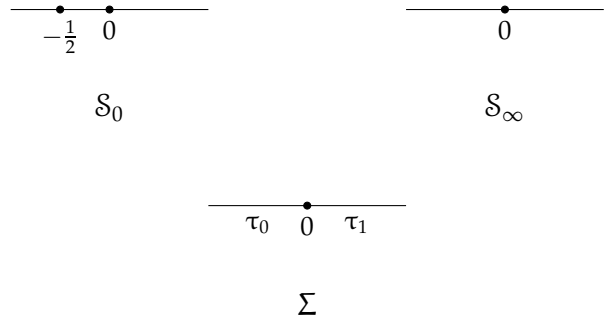
\begin{figure}[H]
\begin{center}
\begin{tikzpicture}[scale = 1.3]
		\draw  (-3,1) -- (-1,1);
		\draw (1,1) -- (3,1);
		\filldraw (-2.5,1) circle (1pt) node[below]{\footnotesize{$-\frac{1}{2}$}};
		\filldraw (-2,1) circle (1pt)node[below]{\footnotesize{$0$}};
		\filldraw (2,1) circle (1pt)node[below]{\footnotesize{$0$}};
		\draw (-2,0) node{$\S_0$};
		\draw (2,0) node{$\S_{\infty}$};
		\draw (-1,-1) -- (1,-1);
		\filldraw (0,-1) circle (1pt)node[below]{\footnotesize{$0$}};
		\draw (-0.5,-1) node[below]{\footnotesize{$\tau_0$}};
		\draw (0.5,-1) node[below]{\footnotesize{$\tau_1$}};
		\draw (0,-2) node{$\Sigma$};
	\end{tikzpicture}\caption{Divisorial fan for $\H_2$ as a $\C^*$-surface}\label{fig:poly-fan-h2}
	\end{center}
\end{figure}
We consider the T-invariant Cartier divisor $D_h$ on $\H_2$ from \cite[Example~3.24]{PS}. This is a very ample divisor whose corresponding Cartier divisorial support function $h= (h_p)_p$ on $\S$ with recession function $\rec(h) = \overline{h}$ is given in the following figure \ref{fig:supp-fct}.
\begin{figure}[H]
\begin{center}
\begin{tikzpicture}[scale = 1.3]
\draw (-3,2) node{$h_0$};
\draw (1.5,2) node{$h_{\infty}$};
\draw (-0.5,-0.5) node{$\overline{h}$};
		\draw  (-3,1) -- (-1,1);
		\draw (1,1) -- (3,1);
		\filldraw (-2.5,1) circle (1pt) node[below]{\footnotesize{$-\frac{1}{2}$}};
		\filldraw (-2,1) circle (1pt)node[below]{\footnotesize{$0$}};
		\filldraw (2,1) circle (1pt)node[below]{\footnotesize{$0$}};
		\draw (-2,0) node{$\S_0$};
		\draw (2,0) node{$\S_{\infty}$};
		\draw (-1,-1.5) -- (1,-1.5);
		\filldraw (0,-1.5) circle (1pt)node[below]{\footnotesize{$0$}};
		\draw (-0.5,-1.5) node[below]{\footnotesize{$\tau_0$}};
		\draw (0.5,-1.5) node[below]{\footnotesize{$\tau_1$}};
		\draw (0,-2.5) node{$\Sigma$};
		
		\draw (-3,1) node[above]{\footnotesize{$u_2+1$}};
		\draw (-2.25,1) node[above]{\footnotesize{$u_1$}};
		\draw (-1.5,1) node[above]{\footnotesize{$u_0$}};
		
		\draw (1.5,1) node[above]{\footnotesize{$u_2-1$}};
		\draw (2.5,1) node[above]{\footnotesize{$u_0-1$}};
		
		\draw (-0.5,-1.5) node[above]{\footnotesize{$u_2$}};
		\draw (0.5,-1.5) node[above]{\footnotesize{$u_0$}};
	\end{tikzpicture}\caption{Divisorial support function $h$}\label{fig:supp-fct}
	\end{center}
\end{figure}
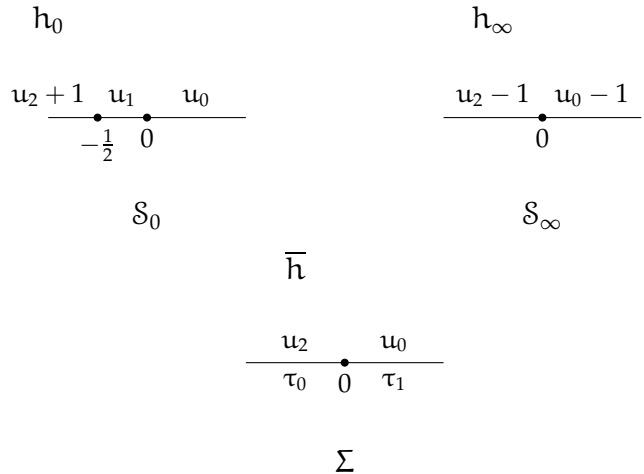
Here, the $u_i \in M$ are given by $u_0\equiv 0$, $u_1 \colon x \mapsto x $, $u_2 \colon x \mapsto 3x$. 

We denote by $v_{0,-\frac{1}{2}}$, $v_{0,0}$ the vertices in $\S_0$ and by $v_{\infty,0}$ the vertex in $\S_{\infty}$. From \eqref{eq:weil-div} we obtain that the Weil divisor $D_h$ corresponding to $h$ can be written as
\[
D_h = 3B_{\tau_0}+Z_{0,-\frac{1}{2}} + Z_{\infty, 0}.
\]
Here, $B_{\tau_0}$ denotes the horizontal divisor corresponding to $\tau_0 \in \Sigma(1)$ and $Z_{0,-\frac{1}{2}}$, respectively $Z_{\infty, 0}$ the vertical divisors corresponding to the vertices $v_{0,-\frac{1}{2}}$ and $v_{\infty,0}$ in $\S_0$ and $\S_{\infty}$, respectively.

It follows from \cite[Proposition 3.31]{PS} that the self-intersection number of $D_h$ can be computed as 
\[
D_h^2 = 2\cdot \on{vol}(h^*),
\]
where $h^*$ is the dual function of $h$, and $\on{vol}$ is some well-defined volume function (we refer to \cite[Section 3]{PS} for details). An easy calculation using the description of $h^*$ given in Figure 3 of \emph{loc.~cit.} and the definition of $\on{vol}$ gives 
\[
D_h^2 = 2\cdot(1+1+3) = 10.
\]
We will now compute $D_h^2$ using Proposition \ref{cor:top-int}. First we compute $c_1 \coloneqq  h \cdot c_{[\H_2]}$.
Using the definition of the intersection pairing in Section \ref{sec:intersect}, we get that $c_1$ is the generalized Minkowski weight on $V_1 \cup R_1$ with values on $V_1$ given by
\[
c_1\left(v_{0,-\frac{1}{2}}\right) = 1, \;c_1\left(v_{0,0}\right) = 1, \; c_1\left(v_{\infty,0}\right) = 3
\]
and with values on $R_1$ given by
\[
c_1\left(\tau_1\right) = 3, \; c_1\left(\tau_0\right) = 2.
\]
It is easy to check that $c_1$ satisfies the balancing conditions from Definition~\ref{def:gen-mink}. Indeed, we only have to consider condition $(2)$ (since conditions $(1)$ and $(3)$ are empty). 

Let $\tau = \{0\} \in \Sigma(0)$. For the first part of condition $(2)$ (Equation \eqref{eq:cond_2.1})  we compute 
\begin{align*}
&\sum_{F \in \S_p(0)}\mu(v_F)v_Fc_1(F) + \sum_{\tau \in \Sigma(1)}v_{\tau}c_1(\tau) \\
 &= -c_1\left(v_{0,-\frac{1}{2}}\right) + c_1\left(\tau_1\right)-c_1\left(\tau_0\right) \\
&= -1+3-2=0.
\end{align*}

Now, to check the second part of condition $(2)$ (Equation \eqref{eq:cond_2.2}) we compute
\begin{align*}
&\sum_{F\in \S_0(0)} \mu(v_F)c_1(F) = 2\cdot c_1\left(v_{0,-\frac{1}{2}}\right) + 1\cdot c_1\left(v_{0,0}\right) = 3 = 1 \cdot c_1\left(v_{\infty,0}\right) =\sum_{F\in \S_{\infty}(0)} \mu(v_F)c_1(F).
\end{align*}
Hence $c_1$ satisfies the balancing conditions from Definition~\ref{def:gen-mink}.

By Proposition \ref{cor:top-int} we get 
\begin{align*}
D_h^2 &= \int_{\on{Tot}(\S)}h\mu_h \\
&= \sum_{F \in \S_p(0)}-h(v_F)\mu_h(F) + \sum_{\tau \in \Sigma(1)}-h_{\tau}(v_{\tau})\mu_h(\tau) \\
&=c_1\left(v_{0,-\frac{1}{2}}\right) +  c_1\left(v_{\infty,0}\right) + 3 \cdot  c_1\left(\tau_0\right) \\
&= 1+3+3\cdot 2 = 10,
\end{align*}
which, as we have seen, is compatible with the results of \cite{PS}.
\end{exa}

\printbibliography

%\newpage \thispagestyle{empty}
%\begin{center}
%\textbf{Data availablity statement}
%\end{center}
%\vspace{2cm}
%
%
%Data sharing is not applicable to this article.
%\newpage \thispagestyle{empty}
%\begin{center}
%\textbf{Conflict of interest statement}
%\end{center}
%\vspace{2cm}
%
%
%Conflict of interest is not applicable to this article.

\end{document}